\documentclass[11pt,reqno]{amsart}
\usepackage{amsmath}
\usepackage{amsthm}
\usepackage{graphicx}
\usepackage{amsfonts}
\usepackage{amssymb}
\usepackage{enumerate}
\usepackage[margin=1.1in]{geometry}
\numberwithin{equation}{section}

\newtheorem{theorem}{Theorem}[section]
\newtheorem{lemma}[theorem]{Lemma}
\newtheorem{proposition}[theorem]{Proposition}
\newtheorem{corollary}[theorem]{Corollary}

\theoremstyle{definition}

\newtheorem{remark}[theorem]{Remark}

\newtheorem{innercustomthm}{Assumption}
\newenvironment{assumption}[1]
  {\renewcommand\theinnercustomthm{#1}\innercustomthm}
  {\endinnercustomthm}

\def\E{{\mathbb E}}
\def\R{{\mathbb R}}

\def\PP{{\mathbb P}}
\def\P{{\mathcal P}}
\def\RC{{\mathcal R}}

\def\V{{\mathcal V}}

\def\Q{{\mathcal Q}}
\def\L{{\mathcal L}}
\def\G{{\mathcal G}}

\def\F{{\mathcal F}}
\def\FF{{\mathbb F}}

\def\C{{\mathcal C}}

\title{Limit theory for controlled McKean-Vlasov dynamics}
\author{Daniel Lacker}
\address{Division of Applied Mathematics, Brown University}
\email{daniel\_lacker@brown.edu}
\thanks{This material is based upon work supported by the National Science Foundation under Award No. DMS 1502980.}

\begin{document}

\begin{abstract}
This paper rigorously connects the problem of optimal control of McKean-Vlasov dynamics with large systems of  interacting controlled state processes. Precisely, the empirical distributions of near-optimal control-state pairs for the $n$-state systems, as $n$ tends to infinity, admit limit points in distribution (if the objective functions are suitably coercive), and every such limit is supported on the set of optimal control-state pairs for the McKean-Vlasov problem. Conversely, any distribution on the set of optimal control-state pairs for the McKean-Vlasov problem can be realized as a limit in this manner. Arguments are based on controlled martingale problems, which lend themselves naturally to existence proofs; along the way it is shown that a large class of McKean-Vlasov control problems admit optimal Markovian controls.
\end{abstract}

\maketitle

\section{Introduction}

The past decade has seen a surge of interest in the optimal control of McKean-Vlasov dynamics, also known as mean field control. This problem can be described loosely as follows: The controller chooses a process $\alpha$, which in turn determines the state process $X$ via a McKean-Vlasov stochastic differential equation (SDE)
\[
dX_t = b(t,X_t,\PP\circ X_t^{-1},\alpha_t)dt + \sigma(t,X_t,\PP\circ X_t^{-1},\alpha_t)dW_t,
\]
where $W$ is a Brownian motion and $\PP \circ X_t^{-1}$ denotes the law of $X_t$. The controller seeks to maximize a functional of the form
\begin{align}
\E\left[\int_0^Tf(t,X_t,\PP\circ X_t^{-1},\alpha_t)dt + g(X_T,\PP\circ X_T^{-1})\right], \label{intro:SDE}
\end{align}
where $T > 0$ is a fixed time horizon.
The unusual feature of this control problem is that the functions $(b,\sigma,f,g)$ depend on the law $\PP\circ X_t^{-1}$ of the state process.

The study of McKean-Vlasov control problems is often justified by a heuristic connection to control problems involving large but finite numbers of interacting state processes. More precisely, imagine there are $n$ state processes interacting through their empirical measures via the following SDE system:
\begin{align}
dX^i_t &= b(t,X^i_t,\widehat{\mu}^n_t,\alpha^i_t)dt + \sigma(t,X^i_t,\widehat{\mu}^n_t,\alpha^i_t)dW^i_t, \label{intro:SDE-n} \\
\widehat{\mu}^n_t &= \frac{1}{n}\sum_{k=1}^n\delta_{X^k_t}. \nonumber
\end{align}
Here $W^1,\ldots,W^n$ are independent Brownian motions, and $\alpha^1,\ldots,\alpha^n$ are controls chosen by a central planner. The objective of this central planner is to maximize the averaged objective
\[
\frac{1}{n}\sum_{i=1}^n\E\left[\int_0^Tf(t,X^i_t,\widehat{\mu}^n_t,\alpha^i_t)dt + g(X^i_T,\widehat{\mu}^n_T)\right].
\]
Of course, if $(X_0^i,W^i,\alpha^i)_{i=1}^n$ are suitably exchangeable (and the SDEs sufficiently well-posed), then each term in the average is equal, and the problem reduces to maximizing the objective corresponding to a single state.

When there is no control present (i.e., no $\alpha$ in $b$ or $\sigma$), it is by now well established that the empirical measure flow $(\widehat{\mu}^n_t)_{t \in [0,T]}$ of \eqref{intro:SDE-n} converges in a sense to the measure flow $(\PP \circ X_t^{-1})_{t \in [0,T]}$ arising from \eqref{intro:SDE}; see, for instance \cite{oelschlagermkv,gartnermkv,sznitman}. This is true, at least, under reasonable continuity assumptions on $(b,\sigma)$, the most important of which is that the interactions are \emph{weak} or \emph{nonlocal} in the sense that the dependence on the measure argument is continuous with respect to weak convergence or a Wasserstein metric.

For the controlled model, however, is not obvious that this limit should commute with the optimization. The primary goal of this paper is to address this issue by providing general conditions under which a sequence of optimizers of the $n$-state system must converge (in the sense of empirical measure) to solutions of the McKean-Vlasov control problem. More precisely, Theorem \ref{th:main-limit} shows, under modest assumptions on $(b,\sigma,f,g)$, that the empirical measure flows $(\widehat{\mu}^n_t)_{t \in [0,T]}$ of optimally controlled $n$-state systems are tight, and every limit in distribution is supported on the set of measures flows $(\PP \circ X_t^{-1})_{t \in [0,T]}$, where $X$ is an optimally controlled state in the McKean-Vlasov control problem. As an immediate corollary, whenever the McKean-Vlasov control problem admits a unique optimal control, we obtain a proper convergence result or \emph{propagation of chaos} \cite{sznitman}.

Our arguments are largely based on martingale problems, combining ideas from the McKean-Vlasov limit theory with a well-established compactification method for stochastic control. The state equations, both for the $n$-state problem and the McKean-Vlasov control problem, are formulated as controlled martingale problems and with \emph{relaxed} (i.e., measure-valued) controls. For standard stochastic control problems, this formulation provides a certain compactness which has facilitated very general results on the existence of optimal controls. This idea originated with Fleming \cite{fleming-generalized} and matured with the works of El Karoui et al. \cite{elkaroui-compactification} and Haussmann-Lepeltier \cite{haussmannlepeltier-existence}, later seeing extensions to general state spaces \cite{kurtzstockbridge-1998}. Our Theorem \ref{th:existence} provides an analogous result on the existence of optimal relaxed controls for the McKean-Vlasov problem. Moreover, as in \cite{elkaroui-compactification,haussmannlepeltier-existence}, we show under an additional convexity hypothesis that there exists an optimal \emph{Markovian} control. Remarkably, the mean field term does not complicate the arguments leading to Markovian controls, which are based on the mimicking theorem of Gy\"ongy \cite{gyongy-mimicking}, or rather the generalization due to Brunick and Shreve \cite{brunickshreve-mimicking}.

The proof of the main limit theorem follows the well trodden path of formulating the limiting equation as a martingale problem in the sense of Stroock and Varadhan \cite{stroockvaradhanbook}, with an additional nonlinearity stemming from the mean field term term. This particular approach to the study of McKean-Vlasov limits seems to have originated with Oelschl\"ager \cite{oelschlagermkv}, while the impressive paper of G\"artner \cite{gartnermkv} contains the most broadly applicable results for models with continuous coefficients. Similar martingale arguments have been applied to a number of related models, including stronger interactions \cite{oelschlager-moderate,meleardroelly-moderate,jourdainmeleard-moderate}, rank-based models \cite{shkolnikov-largesystems,jourdainreygner-rank}, and Boltzmann-type models \cite{meleard1996asymptotic,graham1997stochastic}. The monograph of Sznitman \cite{sznitman} provides a general overview and a bird's eye view of some variants.

Our limit theorem appears to be the first its kind for controlled diffusions, and only three recent papers seem to touch on this: First, Fischer and Livieri \cite{fischer-meanfield} prove a limit theorem for a very special case of our model arising from mean-variance portfolio optimization. Second, Fornasier and Solombrino \cite{fornasier2014mean} treat a general class of related deterministic (i.e., $\sigma \equiv 0$) models; our results allow for degenerate volatility but do not  subsume theirs. Last but not least, Budhiraja et al. \cite{budhirajadupuisfischer} study weak limits of empirical measures of controlled interacting diffusions with relaxed controls, en route to proving a large deviation principle for the McKean-Vlasov limit. Section 5 of their paper contains similar analysis to our Section \ref{se:limits}, but they encounter only particular types of coefficients, with linear-quadratic dependence on the control variable.

The literature on McKean-Vlasov optimal control problems is focused primarily on solution techniques. Only one paper \cite{bahlali2014existence} seems to adopt remotely similar techniques to ours, using relaxed controls (but not martingale problems) and much more restrictive assumptions on the form of the coefficients. The most popular techniques are based on extending Pontryagin's maximum principle \cite{anderssondjehiche-maximum,buckdahndjehicheli-general,carmonadelarue-mkv,bensoussan-mfgbook} or deriving a dynamic programming principle, and with it a form of a Hamilton-Jacobi-Bellman equation on a space of probability measures \cite{laurierepironneau-dynamic,pham-MKVcontrol,bayraktar-cosso-pham} (related to the so-called master equation studied in \cite{carmonadelarue-master,bensoussan-masterequation}). Our solvability result does not provide any insight on how to construct an optimizer, and its strength lies rather in its generality, requiring not even Lipschitz assumptions.

While our assumptions on the model parameters are quite modest, several interesting extensions are left untouched. Most notably, we do not address models with \emph{common noise}, in which an additional independent Brownian motion $B$ appears in the dynamics, and the law $\PP \circ X_t^{-1}$ in the coefficients is replaced by the conditional law $\PP(X_t \in \cdot | B_s, \ s \le t)$. See the recent work of Pham and Wei \cite{pham-stochasticMKVcontrol} for analysis of this model. In another direction, the same authors in \cite{pham-MKVcontrol} study an extension of the basic model in which the coefficients depend on the law of the control, not just the state.

The optimal control of McKean-Vlasov dynamics is closely related to \emph{mean field game theory}, which was developed by Lasry and Lions \cite{lasrylionsmfg} and Huang, Malham\'e, and Caines \cite{huangmfg1}. Mean field games are essentially concerned with the continuum limit of a \emph{competitive} form of the $n$-state control problem, in which the controls $\alpha^1,\ldots,\alpha^n$ are chosen by different agents in Nash equilibrium.  In several applications, in fact, controlled McKean-Vlasov dynamics are studied so that the competitive (decentralized) outcome can be compared with the Pareto optimal (centralized) one \cite{huang-centralized,huang2012social}. The paper \cite{carmonadelaruelachapelle-mkvvsmfg} and the forthcoming book \cite{carmonadelarue-book} study and compare these two distinct regimes, highlighting the significant methodological overlap. It is worth mentioning in particular that martingale methods and relaxed controls have been applied in the study of mean field games, both for existence theory \cite{lacker-mfgcontrolledmartingaleproblems,carmonadelaruelacker-mfgcommonnoise} and limit theory \cite{lacker-mfglimit,fischer-mfgconnection}, and the present work borrows several technical points from these papers. 

The paper is organized as follows.
Section \ref{se:modelsetup} carefully formulates both the McKean-Vlasov and $n$-state control problems, stating all of the main assumptions and results. The remaining sections are devoted to the proofs. Section \ref{se:estimates} derives some preliminary estimates on the state processes, which are put to use Section \ref{se:existenceproofs} to prove the main existence theorems. The proofs of the main limit theorems comprise Sections \ref{se:limits} and \ref{se:limittheorems-proof}. Finally, Section \ref{se:strongvsweak} contains the proof (of Theorem \ref{th:strongequalsweak-MF}) that the optimal value of the control problem is the same for the usual strong formulation and for our preferred relaxed formulation, under suitable assumptions.

\section{Model setup and main results} \label{se:modelsetup}
For a metric space $E$, let $\P(E)$ denote the set of Borel probability measures on $E$, and endow $\P(E)$ with the topology of weak convergence. Fix $p \ge 1$ throughout the paper. For a complete separable metric space $(E,d)$, let $\P^p(E)$ denote the set of $\mu \in \P(E)$ with $\int_Ed(x,x_0)^p < \infty$ for some  $x_0 \in E$. Endow $\P^p(E)$ with the $p$-Wasserstein metric,
\begin{align}
\ell_{E,p}(\mu,\nu) = \inf\left\{\int d^p\,d\pi : \pi \in \P(E \times E) \text{ has marginals } \mu \text{ and } \nu\right\}. \label{def:wasserstein}
\end{align}
As is well known, $\ell_{E,p}(\mu_n,\mu) \rightarrow 0$ if and only if $\int\varphi\,d\mu_n \rightarrow \int\varphi\,d\mu$ for every continuous function $\varphi$ satisfying $|\varphi(x)| \le c(1+d(x,x_0)^p)$ for all $x \in E$, for some $c \ge 0$.
The Borel $\sigma$-field of $\P^p(E)$ is the same as the one induced by the Borel $\sigma$-field of $\P(E)$, which is in turn equivalent to the $\sigma$-field induced by the evaluations $\P^p(E) \ni \mu \mapsto \mu(C)$ for Borel sets $C \subset E$.
For our purposes, the most pertinent topological properties of $\P^p(E)$ are summarized in the appendix of \cite{lacker-mfgcontrolledmartingaleproblems}, but see also \cite[Chapter 7]{villanibook} for more details.

A time horizon $T > 0$ is fixed throughout, along with three exponents $(p',p,p_\sigma)$, an initial distribution $\lambda \in \P(\R^{d})$, and functions
\begin{align*}
(b,\sigma,f) &: [0,T] \times \R^{d} \times \P^p(\R^{d}) \times A \rightarrow \R^d \times \R^{d \times d_W} \times \R, \\
g &: \R^{d} \times \P^p(\R^{d}) \rightarrow \R.
\end{align*}
Here $d$ and $d_W$ denote the respective dimensions of the state and noise processes.
The following standing assumptions, heavily inspired by \cite{lacker-mfgcontrolledmartingaleproblems}, are in force throughout the paper:

\begin{assumption}{\textbf{A}} \label{assumption:A}
{\ }
\begin{enumerate}
\item[(A.1)] $A$ is a closed subset of a Euclidean space.
\item[(A.2)] The exponents satisfy $p' > p \ge 1 \vee p_\sigma$ and $p' \ge 2 \ge p_\sigma \ge 0$, and also $\lambda \in \P^{p'}(\R^{d})$.
\item[(A.3)] The functions $b$ and $\sigma$, are jointly continuous, and $f$ and $g$ are upper semicontinuous.
\item[(A.4)] There exists $c_1 > 0$ such that, for all $(t,x,m,a)$,
\begin{align*}
|b(t,x,m,a)| &\le c_1\left[1 + |x| + \left(\int_{\R^{d}}|z|^pm(dz)\right)^{1/p} + |a|\right], \\
|\sigma(t,x,m,a)|^2 &\le c_1\left[1 + |x|^{p_\sigma} + \left(\int_{\R^{d}}|z|^pm(dz)\right)^{p_\sigma/p} + |a|^{p_\sigma}\right].
\end{align*}
\item[(A.5)] There exist $c_2, c_3 > 0$ such that, for each $(t,x,m,a)$, 
\begin{align*}
g(x,m) &\le c_2\left(1 + |x|^p + \int_{\R^{d}}|z|^pm(dz)\right), \\
g(x,m) &\ge -c_2\left(1 + |x|^{p'} + \int_{\R^{d}}|z|^{p'}m(dz)\right), \\
f(t,x,m,a) &\le c_2\left(1 + |x|^p + \int_{\R^{d}}|z|^pm(dz)\right) - c_3|a|^{p'}, \\
f(t,x,m,a) &\ge -c_2\left(1 + |x|^{p'} + \int_{\R^{d}}|z|^{p'}m(dz) + |a|^{p'}\right).
\end{align*}
\end{enumerate}
\end{assumption}

These minimal assumptions will suffice for an existence theorem. The least innocuous of these is the coercivity assumption (A.5) on the running objective $f$, which is crucial for compactness purposes (see Lemmas \ref{le:estimate-optimal} and \ref{le:estimate-optimal-n}). It should be noted that our methods also apply to a compact control space $A$ and bounded coefficients $(b,\sigma,f,g)$, continuous with respect to weak convergence, and in fact the proofs become significantly simpler in this case. The limit theorems require an additional assumption, mainly for providing uniqueness of the controlled McKean-Vlasov equations:

\begin{assumption}{\textbf{B}} \label{assumption:B}
There exists $c_1' > 0$ such that, for all $(t,x,x',m,m',a)$,
\begin{align*}
|b(t,x,m,a) - b(t,x',m',a)| &+ |\sigma(t,x,m,a) - \sigma(t,x',m',a)| \le c_1'\left(|x-x'| + \ell_{\R^{d},p}(m,m')\right).
\end{align*}
Moreover, the functions $f$ and $g$ are continuous.
\end{assumption}

\subsection{Relaxed controls and canonical spaces} \label{se:relaxed-canonical}
The space $\V$ of \emph{relaxed controls} is defined as the set of measures $q$ on $[0,T] \times A$ with first marginal equal to Lebesgue measure and with
\[
\int_{[0,T] \times A}|a|^pq(dt,da) < \infty.
\]
Noting that each $q \in \V$ has total mass $T$, we may endow $\V$ with a suitable scaling of the $p$-Wasserstein distance. More precisely, equip $\V$ with the metric
\begin{align}
d_\V(q^1,q^2) = \ell_{[0,T] \times A,p}(q^1/T,q^2/T), \label{def:d_V}
\end{align}
where $\ell$ is the Wasserstein distance defined in \eqref{def:wasserstein} relative to the metric on $[0,T] \times A$ given by $((t,a),(t',a')) \mapsto |t-t'| + |a-a'|$.
Each $q \in \V$ is identified with a measurable function $[0,T] \ni t \mapsto q_t \in \P^p(A)$, defined uniquely up to almost sure equality by $q(dt,da)=dtq_t(da)$. Note that $\V$ is a Polish space because $A$ is.
A relaxed control of the form $q(dt,da)=dt\delta_{\alpha(t)}(da)$ for some measurable function $\alpha : [0,T] \rightarrow A$ is called a \emph{strict control}. 
It is known that there exists a version of the map $[0,T] \times \V \ni (t,q) \mapsto q_t \in \P^p(A)$ which is predictable with respect to the filtration $\FF^\Lambda=(\F^\Lambda_t)_{t \in [0,T]}$, where $\F^\Lambda_t$ is generated by the maps $q \mapsto q([0,s] \times C)$, where $s \le t$ and $C \subset A$ is Borel (see, e.g., \cite[Lemma 3.2]{lacker-mfgcontrolledmartingaleproblems}). In particular, this lets us freely identify any random element $\Lambda$ of $\V$ with a corresponding $\P^p(A)$-valued $\FF^\Lambda$-predictable process $(\Lambda_t)_{t \in [0,T]}$.

We will work also with the path space $\C^d = C([0,T];\R^d)$, equipped with the supremum norm $\|x\| = \sup_{t \in [0,T]}|x_t|$. For $m \in \P(\C^d \times \V)$, let $m^x$ denote the $\C^d$-marginal. For $m^x \in \P(\C^d)$ and $t \in [0,T]$, let $m^x_t \in \P(\R^d)$ denote the time-$t$ marginal, i.e., the image of $m^x$ under the map $x \mapsto x_t$. Equip $\C^d\times\V$ with the metric
\begin{align}
d_{\C^d\times\V}((x,q),(x',q')) = \|x-x'\| + d_\V(q,q'), \label{def:productmetric}
\end{align}
where $d_\V$ was defined in \eqref{def:d_V}, and $\|\cdot\|$ is the supremum norm on $\C^d$. Then $\P^p(\C^d\times\V)$ is endowed with the corresponding Wasserstein metric $\ell_{\C^d\times\V,p}$ induced by $d_{\C^d\times\V}$. We will state our main limit theorems in terms of $\P^p(\P^p(\C^d\times\V))$, which is equipped with the Wasserstein metric $\ell_{\P^p(\C^d\times\V),p}$ induced by equipping $\P^p(\C^d\times\V)$ with the metric $\ell_{\C^d\times\V,p}$. 
See again the appendix of \cite{lacker-mfgcontrolledmartingaleproblems} for a more detailed discussion of this topology. For now, simply note that if $P_n \rightarrow P$  in $\P^p(\P^p(\C^d\times\V))$ then a fortiori $P_n \rightarrow P$  in $\P(\P(\C^d\times\V))$ (i.e., weakly).

\subsection{The mean field control problem}

We begin by describing the strong form of the McKean-Vlasov control problem. Suppose we are given a filtered probability space $(\Omega,\F,\FF,\PP)$ supporting a $d_W$-dimensional $\FF$-Wiener process $W$ and an $\F_0$-measurable $\R^d$-valued random variable $\xi$. Here $\FF$ is the augmented filtration generated by the initial state and Wiener process, i.e., the (right-continuous) completion of $(\sigma(\xi,W_s : s \le t))_{t \ge 0}$. An $\FF$-progressively measurable $A$-valued process $\alpha$ is called an \emph{admissible control} if it satisfies
\[
\E\int_0^T|\alpha_t|^pdt < \infty,
\]
and if there exists a unique square-integrable strong solution on $(\Omega,\F,\FF,\PP)$ of the McKean-Vlasov SDE
\[
dX_t = b(t,X_t,\PP \circ X_t^{-1},\alpha_t)dt + \sigma(t,X_t,\PP \circ X_t^{-1},\alpha_t)dW_t, \quad X_0 = \xi.
\]
The strong form of the McKean-Vlasov control problem is to maximize
\[
\E\left[\int_0^Tf(t,X_t,\PP \circ X_t^{-1},\alpha_t)dt + g(X_T,\PP \circ X_T^{-1})\right]
\]
over all admissible controls. Note that an admissible control induces a probability measure $\PP \circ (X,dt\delta_{\alpha_t}(da))^{-1}$ on $\C^d\times\V$. Let $\RC^s$ denote the set of such measures, and refer to an element of $\RC^s$ as a \emph{strong control}. The definition of $\RC^s$ is insensitive to the choice of probability space $(\Omega,\F,\FF,\PP)$, provided that it satisfies the above requirements. Hence, we make no further reference to this particular $(\Omega,\F,\FF,\PP)$.

We next describe the relaxed form of the control problem, abandoning the probability space of the previous paragraph. 
Let $(X,\Lambda)$ denote the projection maps or canonical processes on $\C^d \times \V$. As in Section \ref{se:relaxed-canonical}, we may conflate the random measure $\Lambda(dt,da)$ and the $\P^p(A)$-valued process $(\Lambda_t)_{t \in [0,T]}$. The space $\C^d \times \V$ is equipped with the filtration generated by these canonical processes $(X_t,\Lambda_t)_{t \in [0,T]}$.
Define the generator $\L$ to act on smooth compactly supported functions $\varphi$ by
\begin{align}
\L\varphi(t,x,m,a) = b(t,x,m,a) \cdot \nabla\varphi(x) + \frac{1}{2}\mathrm{Tr}\left[\sigma\sigma^\top(t,x,m,a)\nabla^2\varphi(x)\right], \label{def:generator}
\end{align}
where $\nabla$ and $\nabla^2$ denote the gradient and Hessian, respectively. That is, $\L$ maps a function of $\R^d$ to a function of $[0,T] \times \R^d \times \P^p(\R^d) \times A$.
Define $\RC$ to be the set of $m \in \P^p(\C^d \times \V)$ such that $m^x_0 = \lambda$ (recalling that $m^x_t := m \circ X_t^{-1}$) and the process 
\[
\varphi(X_t) - \int_0^t\int_A\L\varphi(s,X_s,m^x_s,a)\Lambda_s(da)dt
\]
is a $m$-martingale for every smooth compactly supported $\varphi$. It is a straightforward consequence of It\^o's formula that $\RC^s \subset \RC$; that is, a strong control induces a relaxed control. An element of $\RC$ is called a \emph{control}, or a \emph{relaxed control} for emphasis. We say a control $m \in \RC$ is \emph{Markovian} if there exists a measurable map $\hat{\alpha} : [0,T] \times \R^d \rightarrow A$ such that $m(\Lambda_t = \delta_{\hat{\alpha}(t,X_t)}, \ a.e. \ t \in [0,T])=1$.
See Remark \ref{re:markovian} below for a caveat regarding this use of the term \emph{Markovian}.

\begin{remark} \label{re:rcnonempty}
Under assumption \ref{assumption:A}, the set $\RC$ is nonempty. In particular, for any fixed constant control $a_0 \in A$, there exists $m \in \RC$ such that $m(\Lambda_t = \delta_{a_0}, \text{ a.e. } t \in [0,T]) = 1$. Indeed, this follows from the results of G\"artner \cite[Theorem 2.9]{gartnermkv} or Funaki \cite[Theorem 2.1]{funaki1984certain} on McKean-Vlasov equations. When both assumptions \ref{assumption:A} and \ref{assumption:B} hold, $\RC^s$ is also nonempty, by standard Lipschitz arguments (see \cite[Section I]{sznitman} or the proof of \cite[Lemma 3.1]{funaki1984certain}).
\end{remark}

The McKean-Vlasov control problem, in relaxed form, is to maximize 
\begin{align}
\Gamma(m) := \E^m\left[\int_0^T\int_Af(t,X_t,m^x_t,a)\Lambda_t(da)dt + g(X_T,m^x_T)\right] \label{def:Gamma}
\end{align}
over all choices of $m \in \RC$. Note that $\Gamma : \P^p(\C^d\times\V) \rightarrow \R \cup \{-\infty\}$ is well-defined because of assumption (A.5). Let $\RC^*$ denote the set of optimal controls, i.e., the set of $m \in \RC$ for which $\Gamma(m) \ge \Gamma(\tilde{m})$ for all $\tilde{m} \in \RC$. We are now ready to state the main existence results, with proofs deferred to Sections \ref{se:limittheorems-proof}.

\begin{theorem} \label{th:existence}
Under assumption \ref{assumption:A}, the set $\RC^*$ is nonempty. That is, there exists an optimal relaxed control.
\end{theorem}

Next, we state an existence result for Markovian optimal controls, under an additional assumption, familiar in the control theory literature from the work of Filippov \cite{filippov-convexity}.

\begin{assumption}{\textbf{C}} \label{assumption:C}
For each $(t,x,m) \in [0,T] \times \R^d \times \P^p(\R^d)$, the following set is convex:
\[
K(t,x,\mu) := \left\{\left(b(t,x,m,a),\sigma\sigma^\top(t,x,m,a),z\right) : a \in A, \ z \le f(t,x,m,a)\right\} \subset \R^d \times \R^{d \times d} \times \R.
\]
\end{assumption}

\begin{theorem} \label{th:markovian}
Suppose assumptions \ref{assumption:A} and \ref{assumption:C} hold. Then, for each $m \in \RC$, there exists a Markovian control $\widetilde{m} \in \RC$ satisfying $\widetilde{m}^x_t= m^x_t$ for all $t \in [0,T]$ as well as $\Gamma(\widetilde{m}) \ge \Gamma(m)$. In particular, there exists an optimal Markovian control.
\end{theorem}

The next theorem states that the optimal value of the strong and relaxed formulations are the same. The additional assumptions are minor and can likely be dispensed with. The proof, deferred to Section \ref{se:strongvsweak}, requires some delicate approximations of martingale measures, for which we employ a result of M\'el\'eard \cite{meleard1992martingale}. An alternative proof is possible under less easily verifiable weak uniqueness assumptions, by adapting the methods of \cite{elkaroui-compactification,karouitan-capacities}.

\begin{theorem} \label{th:strongequalsweak-MF}
Suppose assumptions \ref{assumption:A} and \ref{assumption:B} hold. Assume also that $p'=2$ and that the initial condition $\lambda$ satisfies $\int|x|^{p''}\lambda(dx) < \infty$ for some $p'' > p'$. Then the optimal value of the McKean-Vlasov control problem is the same for both the relaxed and strong formulations. That is,
\[
\sup_{m \in \RC^s}\Gamma(m) = \sup_{m \in \RC}\Gamma(m).
\]
\end{theorem}

Given the unusual nature of the martingale problems defining $\RC$, it may be unclear to the uninitiated reader what exactly we have proven to exist in Theorems \ref{th:existence} and \ref{th:markovian}.
This is clarified by the following proposition, which will be useful in the proofs as well.
When the control is present in the volatility, the most useful SDE representation involves martingale measures. Only the very basics of the theory of martingale measures are needed, and these facts are reviewed as they come up. All of the relevant definitions and results are concisely summarized in \cite{elkarouimeleard-martingalemeasure}, but refer to the original monograph of Walsh \cite{walsh-introspde} for a more thorough treatment. When the control is absent from the volatility $\sigma$, the martingale measure $N(da,dt)$ in the following proposition can be replaced with a Wiener process $dW_t$:

\begin{proposition} [Theorem IV-2 of \cite{elkarouimeleard-martingalemeasure}] \label{pr:SDErepresentation}
The set $\RC$ is precisely the set of laws $\PP \circ (X,\Lambda)^{-1}$, where:
\begin{enumerate}
\item $(\Omega,\F,\FF,\PP)$ is a filtered probability space supporting a $d$-dimensional adapted process $X$, a $\P^p(A)$-valued predictable process $\Lambda$, and a (column) vector $N=(N^1,\ldots,N^{d_W})$ of orthogonal $\FF$-martingale measures on $A \times [0,T]$, each with intensity measure $\Lambda_t(da)dt$.
\item $\PP \circ X_0^{-1} = \lambda$.
\item $\E\left[\sup_{t \in [0,T]}|X_t|^p + \int_0^T\int_A|a|^p\Lambda_t(da)dt\right] < \infty$.
\item The state equation holds:
\[
dX_t = \int_Ab(t,X_t,\PP \circ X_t^{-1},a)\Lambda_t(da)dt + \int_A\sigma(t,X_t,\PP \circ X_t^{-1},a)N(da,dt).
\]
\end{enumerate}
The set of Markovian controls is precisely the set of laws $\PP \circ (X,\Lambda)^{-1}$, where:
\begin{enumerate}
\item $(\Omega,\F,\FF,\PP)$ is a filtered probability space supporting a $d$-dimensional adapted process $X$ and a $d_W$-dimensional $\FF$-Wiener process $W$.
\item $\PP \circ X_0^{-1} = \lambda$.
\item There exists a measurable function $\hat{\alpha} : [0,T] \times \R^d \rightarrow A$ such that:
\begin{enumerate}
\item The state equation holds,
\begin{align}
dX_t = b(t,X_t,\PP \circ X_t^{-1},\hat{\alpha}(t,X_t))dt + \sigma(t,X_t,\PP \circ X_t^{-1},\hat{\alpha}(t,X_t))dW_t, \label{def:markov-sde}
\end{align}
\item $\E\left[\sup_{t \in [0,T]}|X_t|^p + \int_0^T|\hat{\alpha}(t,X_t)|^pdt\right] < \infty$.
\item $\Lambda_t = \delta_{\hat{\alpha}(t,X_t)}$ for a.e. $t \in [0,T]$, almost surely.
\end{enumerate}
\end{enumerate}
\end{proposition}

\begin{remark} \label{re:markovian}
A word of caution about the terminology: Under a Markovian control $m \in \RC$, it is not true in general that the state process $X$ is a Markov process. Only when the state equation \eqref{def:markov-sde} is sufficiently well-posed is $X$ truly a Markov process. For instance, letting $(\Omega,\F,\FF,\PP,X,\hat{\alpha})$ be as in the second part of Proposition \ref{pr:SDErepresentation}, define new coefficients $\hat{b}(t,x) = b(t,x,\PP \circ X_t^{-1},\hat{\alpha}(t,x))$ and $\hat{\sigma}(t,x) = \sigma(t,x,\PP \circ X_t^{-1},\hat{\alpha}(t,x))$. If the martingale problem associated to $(\hat{b},\hat{\sigma})$ is well-posed, then the unique in law weak solution $X$ is a Feller process \cite[Chapter 12]{stroockvaradhanbook}.
\end{remark}

\subsection{$n$-state control problems}

This section states the main results on how the McKean-Vlasov control problem arises from $n$-state control problems as $n \rightarrow \infty$.
Assume throughout this section that both assumptions \ref{assumption:A} and \ref{assumption:B} are in force.

We first describe a standard \emph{strong formulation} of the control problems.
Let $(\Omega,\F,\FF,\PP)$ be a filtered probability space supporting independent $d_W$-dimensional $\FF$-Wiener processes $W^1,\ldots,W^n$ as well as i.i.d. $\F_0$-measurable $\R^d$-valued random variables $\xi^1,\ldots,\xi^n$ with law $\lambda$. Assume the filtration $\FF$ is generated by these initial states and Wiener processes, i.e., the (right-continuous) completion of $(\sigma(\xi^1,\ldots,\xi^n,W^1_s,\ldots,W^n_s : s \le t))_{t \ge 0}$. An $\FF$-progressively measurable $A^n$-valued process $(\alpha^1,\ldots,\alpha^n)$ is called am \emph{admissible control} if
\[
\E\int_0^T|\alpha^k_t|^pdt < \infty, \ k=1,\ldots,n,
\]
and if there exists a unique square-integrable strong solution on $(\Omega,\F,\FF,\PP)$ of the SDE system
\begin{align*}
dX^i_t &= b(t,X^i_t,\widehat{\mu}^n_t,\alpha^i_t)dt + \sigma(t,X^i_t,\widehat{\mu}^n_t,\alpha^i_t)dW^i_t, \\ 
\widehat{\mu}^n_t &= \frac{1}{n}\sum_{k=1}^n\delta_{X^k_t}.
\end{align*}
The strong form of the $n$-state control problem is to maximize
\begin{align*}
\frac{1}{n}\sum_{i=1}^n\E\left[\int_0^Tf(t,X^i_t,\widehat{\mu}^n_t,\alpha^i_t)dt + g(X^i_T,\widehat{\mu}^n_T)\right]
\end{align*}
over all admissible controls.
Note that an admissible control induces a probability measure $\PP \circ ((X^i,dt\delta_{\alpha^i_t}(da))_{i=1}^n)^{-1}$ on $(\C^d\times\V)^n$. Let $\RC_n^s$ denote the set of such measures, and refer to an element of $\RC_n^s$ as a \emph{strong control}. As in the previous section, the definition of $\RC^s_n$ is insensitive to the choice of probability space $(\Omega,\F,\FF,\PP)$, provided that it satisfies the above requirements. Hence, we make no further reference to this particular $(\Omega,\F,\FF,\PP)$.

The relaxed form of the $n$-state control problem is defined by working with relaxed controls and weak solutions of the SDEs. Let $(X^i,\Lambda^i)_{i=1}^n$ denote the canonical process on $(\C^d \times \V)^n$. Define the empirical measures
\begin{align}
\widehat{\mu}^n &= \frac{1}{n}\sum_{k=1}^n\delta_{(X^k,\Lambda^k)}, \quad\quad\quad\quad
\widehat{\mu}^{n,x}_t = \frac{1}{n}\sum_{k=1}^n\delta_{X^k_t}, \ t \in [0,T].
\end{align}
Define $\RC_n$ as the set of laws $P \in \P((\C^d\times\V)^n)$ under which $(X^i_0)_{i=1}^n$ are i.i.d. with law $\lambda$, and 
\[
\varphi(X^1_t,\ldots,X^n_t) - \sum_{i=1}^n\int_0^t\int_A\L^n_i\varphi(s,X^1_s,\ldots,X^n_s,a)\Lambda^i_t(da)dt,
\]
is a martingale, where the generator $\L^n_i$ acts on smooth compactly supported functions $\varphi : (\R^d)^n \rightarrow \R$ by
\begin{align}
\L^n_i\varphi(t,x_1,\ldots,x_n,a) = \;&b\left(t,x_i,\frac{1}{n}\sum_{k=1}^n\delta_{x_k},a\right) \cdot \nabla_i\varphi(x_1,\ldots,x_n) \nonumber \\
	& + \frac{1}{2}\mathrm{Tr}\left[\sigma\sigma^\top\left(t,x_i,\frac{1}{n}\sum_{k=1}^n\delta_{x_k},a\right)\nabla_i^2\varphi(x_1,\ldots,x_n)\right], \label{def:generator-n}
\end{align}
where $\nabla_i$ and $\nabla_i^2$ denote the gradient and Hessian with respect to the $i^\text{th}$ variable. An element of $\RC_n$ is called a \emph{control}, or an \emph{(n-state) relaxed control} for emphasis. It is a straightforward consequence of It\^o's formula that $\RC^s_n \subset \RC_n$; that is, a strong control induces a relaxed control. 
We have the following analog of Proposition \ref{pr:SDErepresentation}, and again note that there is no need for martingale measures when $\sigma$ is uncontrolled:

\begin{proposition} [Theorem IV-2 of \cite{elkarouimeleard-martingalemeasure}] \label{pr:SDErepresentation-n}
The set $\RC_n$ equals the set of laws $\PP \circ ((X^i,\Lambda^i)_{i=1}^n)^{-1}$, where:
\begin{enumerate}
\item $(\Omega,\F,\FF,\PP)$ is a filtered probability space supporting $n$ adapted $d$-dimensional processes $X^1,\ldots,X^n$, $n$ predictable $\P^p(A)$-valued processes $\Lambda^1,\ldots,\Lambda^n$, and $nd_W$ orthogonal $\FF$-martingale measures $(N^{i,j})$, for $i=1,\ldots,n$ and $j=1,\ldots,d_W$, where $N^{i,j}$ has intensity $\Lambda^i_t(da)dt$ for each $j$.
\item $X^1_0,\ldots,X^n_0$ are i.i.d. with law $\lambda$.
\item $\E\left[\sup_{t \in [0,T]}|X^i_t|^p + \int_0^T\int_A|a|^p\Lambda^i_t(da)dt\right] < \infty$, for each $i=1,\ldots,n$.
\item The state equation holds, with $N^i=(N^{i,1},\ldots,N^{i,d_W})^\top$:
\begin{align*}
dX^i_t &= \int_Ab(t,X^i_t,\widehat{\mu}^{n,x}_t,a)\Lambda^i_t(da)dt + \int_A\sigma(t,X^i_t,\widehat{\mu}^{n,x}_t,a)N^i(da,dt), \\
\widehat{\mu}^{n,x}_t &= \frac{1}{n}\sum_{k=1}^n\delta_{X^k_t}.
\end{align*}
\end{enumerate}
\end{proposition}

\begin{remark} \label{re:rcn-nonempty}
Under assumption \ref{assumption:A}, the set $\RC_n$ is nonempty for each $n$. In particular, for any fixed constant control $(a^1_0,\ldots,a^n_0) \in A^n$, there exists $P \in \RC_n$ such that, almost surely under $P_n$, $\Lambda^k_t = \delta_{a^k_0}$ for a.e. $t \in [0,T]$ and for each $k=1,\ldots,n$. Indeed, this follows essentially from the classical existence results for martingale problems of Stroock and Varadhan \cite{stroockvaradhanbook}.
\end{remark}

Recalling the definition of $\Gamma$ from \eqref{def:Gamma}, the reward of $P \in \RC_n$ is given by
\[
\E^P[\Gamma(\widehat{\mu}^n)] = \frac{1}{n}\sum_{i=1}^n\E\left[\int_0^T\int_Af(t,X^i_t,\widehat{\mu}^{n,x}_t,a)\Lambda^i_t(da)dt + g(X^i_T,\widehat{\mu}^{n,x}_T)\right].
\]
Note that $\E^P[\Gamma(\widehat{\mu}^n)]$ is well-defined in $[-\infty,\infty)$ for $P \in \P^p((\C^d\times\V)^n)$, thanks to assumption (A.5).
Given $\epsilon \ge 0$, we say $P \in \RC_n$ is a \emph{(relaxed) $n$-state $\epsilon$-optimal control} if
\[
\E^P[\Gamma(\widehat{\mu}^n)] \ge \sup_{Q \in \RC_n}\E^Q[\Gamma(\widehat{\mu}^n)] - \epsilon.
\]
If $\epsilon=0$, we simply say $P$ is a \emph{(relaxed) $n$-state optimal control}.
The following result, at the present level of generality, is due to Haussmann and Lepeltier:

\begin{theorem}[Theorem 4.7 of \cite{haussmannlepeltier-existence}] \label{th:n-player-existence}
Under assumption \ref{assumption:A}, for each $n$, there exists an optimal $n$-state relaxed control.
\end{theorem}

In many cases, the optimal value of the relaxed control problem is the same as that of the strong formulation. Theorem \ref{th:strongequalsweak-n} below is a step in this direction, analogous to Theorem \ref{th:strongequalsweak-MF}. It is nearly a special case of the results of \cite[Section 4]{elkaroui-compactification} when $(b,\sigma,f,g)$ are bounded.

\begin{theorem} \label{th:strongequalsweak-n}
Suppose assumptions \ref{assumption:A} and \ref{assumption:B} hold with $p'=2$. Assume the initial condition $\lambda$ belongs to $\P^{p''}(\R^d)$ for some $p'' > 2$. Then, for every relaxed control $P \in \RC_n$, there exists a sequence of strong controls $P_n \in \RC_n^s$ with $P_n \rightarrow P$ in $\P^p((\C^d\times\V)^n)$ and $\E^{P_n}[\Gamma(\hat{\mu}^n)] \rightarrow \E^{P}[\Gamma(\hat{\mu}^n)]$.
In particular, the optimal value of the $n$-state control problem is the same for both the relaxed and strong formulations; that is,
\[
\sup_{P \in \RC_n}\E^P[\Gamma(\widehat{\mu}^n)] = \sup_{P \in \RC^s_n}\E^P[\Gamma(\widehat{\mu}^n)].
\]
\end{theorem}

\subsection{The main limit theorems}
Now that we understand the structure of the $n$-state and McKean-Vlasov control problems, we are ready to state the main results of the paper. Refer to Section \ref{se:relaxed-canonical} for a discussion of convergence in the space $\P^p(\P^p(\C^d\times\V))$. Recall that $\RC^* \subset \P^p(\C^d\times\V)$ denotes the set of (relaxed) optimal McKean-Vlasov controls.

\begin{theorem} \label{th:main-limit}
Suppose assumptions \ref{assumption:A} and \ref{assumption:B} hold.
For each $n$, let $P_n \in \RC_n$ be a relaxed $n$-state $\epsilon_n$-optimal control, for some sequence $\epsilon_n \rightarrow 0$. Then $(P_n \circ (\widehat{\mu}^n)^{-1})$ is precompact in $\P^p(\P^p(\C^d \times \V))$, and every limit is supported on $\RC^*$.
In particular, 
\begin{align}
\lim_{n\rightarrow\infty}P_n\left(\ell_{\C^d\times\V,p}(\widehat{\mu}^n,\RC^*) \ge \epsilon\right) = 0, \text{ for all } \epsilon > 0. \label{def:distance-limit}
\end{align}
\end{theorem}

\begin{theorem} \label{th:converse-limit}
Suppose assumptions \ref{assumption:A} and \ref{assumption:B} hold.
Let $P \in \P(\P^p(\C^d \times \V))$ be supported on $\RC^*$. Then there exist $\epsilon_n \rightarrow 0$ and a sequence of relaxed $n$-state $\epsilon_n$-optimal controls $P_n \in \RC_n$ such that $P_n \circ (\widehat{\mu}^n)^{-1} \rightarrow P$. Moreover, if $p'=2$ and if $\lambda \in \P^{p''}(\R^d)$ for some $p'' >2$, then the controls $P_n$ can be taken to be strong.
\end{theorem}

If it happens that the optimal McKean-Vlasov control is unique, then an immediate corollary of Theorem \ref{th:main-limit} is a proper convergence theorem, stating that (the empirical measure sequence of) any $n$-state near-optimal $n$-state controls converge in probability to the unique optimal McKean-Vlasov control. An immediate corollary of Theorems \ref{th:main-limit} and \ref{th:markovian} is the following statement, for which $C([0,T];\P^p(\R^d))$ is endowed with the metric $(\mu,\nu) \mapsto \sup_{t \in [0,T]}\ell_{\R^d,p}(\mu_t,\nu_t)$.

\begin{corollary}
Suppose assumptions \ref{assumption:A}, \ref{assumption:B}, and \ref{assumption:C} hold.
For each $n$, let $P_n$ denote a weak $n$-state $\epsilon_n$-optimal control, for some sequence $\epsilon_n \rightarrow 0$. Then $(P_n \circ (\widehat{\mu}^{n,x}_t)_{t \in [0,T]}^{-1})$ is precompact in $\P^p(C([0,T];\P^p(\R^d)))$, and every weak limit is supported on the set $\{(m^x_t)_{t \in [0,T]} : m \in \RC^* \text{ is Markovian}\}$.
\end{corollary}

The rest of the paper is devoted to the proofs. Section \ref{se:estimates} is devoted to some useful preliminary results, including moment estimates on the state process and some continuity properties of the objective functional $\Gamma$. Section \ref{se:existenceproofs} proves the existence theorems \ref{th:existence} and \ref{th:markovian}. Most of the work toward Theorems \ref{th:main-limit} and \ref{th:converse-limit} is done in the preparatory Section \ref{se:limits}, with the main line of the proofs deferred to Section \ref{se:limittheorems-proof}. Finally, Section \ref{se:strongvsweak} discusses Theorems \ref{th:strongequalsweak-MF} and \ref{th:strongequalsweak-n}.

\section{Some first estimates} \label{se:estimates}

This section collects the essential estimates needed in the proofs of almost all of the theorems of the previous section. The first estimates of Section \ref{se:estimates-MF} are in the mean field regime, whereas the estimates of Section \ref{se:estimates-n} pertain to the $n$-state games. The primary role of these estimates is in obtaining compactness. If the control space $A$ were assumed compact, and if the coefficients $b$, $\sigma$, $f$, and $g$ were assumed to be bounded, none of these estimates would be needed. In the following, let $\|x\|_t = \sup_{s \in [0,t]}|x_s|$ denote the truncated supremum norm for $x \in \C^d$, for $t \in [0,T]$, and recall that $\|\cdot\| = \|\cdot\|_T$.

\subsection{Mean field estimates} \label{se:estimates-MF}

The first lemma, stated without proof, is borrowed from \cite{lacker-mfgcontrolledmartingaleproblems}. The second shows how to use the coercivity assumption (A.5) to translate optimality properties into moment bounds.

\begin{lemma}[Lemma 4.3 of \cite{lacker-mfgcontrolledmartingaleproblems}]\label{le:estimate-MF}
For each $\gamma \ge p$ such that $\int|x|^\gamma\lambda(dx) < \infty$, there exists a constant $C \ge 1$, depending only on $\gamma$, $p$, $p'$, $T$, the initial law $\lambda$, and the constant $c_1$ of Assumption (A.4) such that for all $m \in \RC$ we have
\begin{align*}
\int_{\C^d}\|x\|^pm^x(dx) = \E^m\left[\|X\|^\gamma\right] \le C\E^m\left[1 + \int_0^T\int_A|a|^\gamma\Lambda_t(da)dt\right].
\end{align*}
\end{lemma}

\begin{lemma} \label{le:estimate-optimal}
Let $\epsilon > 0$. Suppose $\RC^\epsilon$ is the set of $m \in \RC$ satisfying
\[
\Gamma(m) \ge \sup_{m' \in \RC}\Gamma(m') - \epsilon.
\]
Then
\begin{align}
\sup_{m \in \RC^\epsilon}\E^m\int_0^T\int_A|a|^{p'}\Lambda_t(da)dt < \infty. \label{def:le:estimate-optimal}
\end{align}
Moreover, $\sup_{m \in \RC}\Gamma(m) < \infty$.
\end{lemma}
\begin{proof}
Recall first that $\E^m\int_0^T\int_A|a|^p\Lambda_t(da)dt < \infty$ for all $m \in \RC$ by assumption; this ensures that the following expressions are well-defined.
Use the upper bounds on $f$ and $g$ from assumption (A.5) along with Lemma \ref{le:estimate-MF} to find a constant $C > 0$ (which will change from line to line) such that, for all $m \in \RC$,
\begin{align}
\Gamma(m) &\le C\E^m\left[1 + \|X\|^p + \int_{\C^d}\|x\|^pm^x(dx)\right] - c_3\E^m\int_0^T\int_A|a|^{p'}\Lambda_t(da)dt \nonumber \\
	&\le C\E^m\left[1 + \int_0^T\int_A|a|^p\Lambda_t(da)dt\right] - c_3\E^m\int_0^T\int_A|a|^{p'}\Lambda_t(da)dt. \label{pf:control-estimate1}
\end{align}
This already proves $\sup_{m \in \RC}\Gamma(m) < \infty$, as $a \mapsto C|a|^p - c_3|a|^{p'}$ is bounded from above. To prove the first claim, fix arbitrarily a constant control $a_0 \in A$, and let $m^0$ denote an element of $\RC$ satisfying $m^0(\Lambda_t = \delta_{a_0}, \ a.e. \ t)=1$. (As in Remark \ref{re:rcnonempty}, the existence of such an $m^0$ follows from a result of G\"artner \cite[Theorem 2.9]{gartnermkv}.)
Lemma \ref{le:estimate-MF} implies
\[
\int_{\C^d}\|x\|^{p'}(m^0)^x(dx) = \E^{m^0}\left[\|X\|^{p'}\right] \le C(1 + T|a_0|^{p'}).
\]
Then use the lower bounds of assumption (A.5) to show $\Gamma(m^0) > -\infty$. For $m \in \RC^\epsilon$ we have $\Gamma(m) \ge \Gamma(m^0) - \epsilon$, which combined with \eqref{pf:control-estimate1} yields
\begin{align*}
\sup_{m \in \RC^\epsilon}\E^m\int_0^T\int_A\left(|a|^{p'}- C|a|^p\right)\Lambda_t(da)dt < \infty.
\end{align*}
This is enough to complete the proof.
\end{proof}

\subsection{$n$-state estimates} \label{se:estimates-n}

Here we derive an analogous pair of lemmas for the $n$-state control problem after first recalling some basic facts about martingale measures, all of which can be found in both \cite{walsh-introspde} and \cite{elkarouimeleard-martingalemeasure}. Suppose $N$ is martingale measure with intensity measure $\Lambda_t(da)dt$, where $\Lambda$ is a predictable $\P(A)$-valued process, and $(h^i_t)_{t \in [0,T]}$ is a predictable random function on $A$. \footnote{Assuming $(\Omega,\F,\FF,\PP)$ is the filtered probability space in the background, a \emph{predictable random function on $A$} is a map $h : [0,T] \times \Omega \times A \rightarrow \R$ which is jointly measurable with respect to the $\FF$-predictable $\sigma$-field on $[0,T]\times \Omega$ and the Borel $\sigma$-field on $A$. We suppress $\omega$ from the notation as usual, writing $h_t(a)$ in place of $h(t,\omega,a)$.} If
\[
\E\left[\int_0^T\int_A|h^i_t(a)|^2\Lambda_t(da)dt\right] < \infty,
\]
then
\[
M^i_t = \int_0^t\int_Ah^i_s(a)N(da,ds)
\]
is a martingale, for $i=1,2$. The covariation between $M^1$ and $M^2$ is
\[
[M^1,M^2]_t = \int_0^t\int_Ah^1_s(a)h^2_s(a)\Lambda_s(da)ds,
\]
assuming this integral is well defined.
In particular, the quadratic variation of $M^1$ is
\[
[M^1,M^1]_t = \int_0^t\int_A|h^1_s(a)|^2\Lambda_s(da)ds,
\]
assuming the latter is finite almost surely. 
If $\widetilde{N}$ and $N$ are \emph{orthogonal} martingale measures, then the martingales
\[
\int_0^t\int_Ah_s(a)N(da,ds) \quad\quad\quad \text{and} \quad\quad\quad \int_0^t\int_A\widetilde{h}_s(a)\widetilde{N}(da,ds)
\]
are themselves orthogonal (i.e., the covariation is identically zero) for square-integrable $h,\widetilde{h}$.

\begin{lemma} \label{le:estimate-n}
For each $\gamma \ge p$ such that $\int|x|^\gamma\lambda(dx) < \infty$, there exists a constant $C \ge 1$, depending only on $\gamma$, $p$, $p'$, $T$, the initial law $\lambda$, and the constant $c_1$ of Assumption (A.4) such that, for all $n \ge 1$ and $P \in \RC_n$, we have
\begin{align}
\E^P[\|X^k\|^\gamma] &\le C\E^P\left[1 + \int_0^T\int_A|a|^\gamma\Lambda^k_t(da)dt + \frac{1}{n}\sum_{i=1}^n\int_0^T\int_A|a|^\gamma\Lambda^i_t(da)dt\right] \label{def:estimate-n-1}
\end{align}
for each $k=1,\ldots,n$, and
\begin{align}
\E^P\left[\int_{\C^d}\|x\|^\gamma\widehat{\mu}^{n,x}(dx)\right] &= \frac{1}{n}\sum_{k=1}^n\E^P[\|X^k\|^\gamma] \le C\E^P\left[1 + \frac{1}{n}\sum_{k=1}^n\int_0^T\int_A|a|^\gamma\Lambda^k_t(da)dt\right]. \label{def:estimate-n-2}
\end{align}
\end{lemma}
\begin{proof}
Fix $n$ and $P \in \RC_n$, and use Proposition \ref{pr:SDErepresentation-n} to express $(X^1,\ldots,X^n)$ (under $P$) as the solution of an SDE driven by $(N^1,\ldots,N^n)$, where each $N^k$ is a vector of $d_W$ orthogonal martingale measures with common intensity measure $\Lambda^k_t(da)dt$. That is,
\[
dX^k_t = \int_Ab(t,X^k_t,\widehat{\mu}^{n,x}_t,a)\Lambda^k_t(da)dt + \int_A\sigma(t,X^k_t,\widehat{\mu}^{n,x}_t,a)N^k(da,dt).
\]
Recall the above remarks on quadratic variations of stochastic integrals with respect to martingale measures. Apply the Burkholder-Davis-Gundy inequality and assumption (A.4) to find a universal constant $C$ (which will change from line to line) such that, for each $k=1,\ldots,n$,
\begin{align}
\E\left[\|X^k\|_t^\gamma\right] &\le C\E\left[|X^k_0|^\gamma + \left(\int_0^t\int_A|b(t,X^k_s,\widehat{\mu}^{n,x}_s,a)|\Lambda^k_s(da)ds\right)^\gamma \right] \nonumber \\
	&\quad + C\E\left[ \left(\int_0^t\int_A|\sigma(t,X^k_s,\widehat{\mu}^{n,x}_s,a)|^2\Lambda^k_s(da)ds\right)^{\gamma/2}\right] \nonumber \\
	&\le C\E\left\{1 + |X^k_0|^\gamma + \int_0^t\left[\|X^k\|_s^\gamma + \left(\int_{\C^d}\|x\|_s^p\widehat{\mu}^{n,x}(dx)\right)^{\gamma/p} + \int_A|a|^\gamma\Lambda^k_s(da)\right]ds \right\} \nonumber \\
	&\quad + C\E\left\{\left[\int_0^t\int_A\left(\|X^k\|_s^{p_\sigma} + \left(\int_{\C^d}\|x\|_s^p\widehat{\mu}^{n,x}(dx)\right)^{p_\sigma/p} + |a|^{p_\sigma/p}\right)\Lambda^k_s(da)ds\right]^{\gamma/2}\right\} \nonumber \\
	&\le C\E\left\{1 + \int_0^t\left[\|X^k\|_s^\gamma + \int_{\C^d}\|x\|_s^\gamma \widehat{\mu}^{n,x}(dx) + \int_A|a|^\gamma\Lambda^k_s(da)\right]ds \right\}. \label{pf:estimate-n1}
\end{align}
The derivation of the last line used a number of facts. First of all, note that $\E[|X^1_0|^p]=\E[|X^k_0|^p]$ by symmetry, and this term was subsumed in the constant $C$ in the last line. Second, Jensen's inequality yielded $\left(\int_{\C^d}\|x\|_s^p\widehat{\mu}^{n,x}(dx)\right)^{\gamma/p} \le \int_{\C^d}\|x\|_s^\gamma \widehat{\mu}^{n,x}(dx)$, because $\gamma \ge p$. Finally, to deal with the exponent of $\gamma/2$ outside of the integral, there are two cases. First, if $\gamma \ge 2$,  Jensen's inequality lets us bring the $\gamma/2$ inside of the time integral, and we then use the inequality $|x|^{p_\sigma\gamma/2} \le 1+ |x|^\gamma$ which holds because $p_\sigma \le 2$. The other alternative is $2 > \gamma \ge p \ge 1 \vee p_\sigma$, in which case we use the inequalities  $|x|^{\gamma/2} \le 1 + |x|$ and $|x|^{p_\sigma} \le 1 + |x|^\gamma$. With \eqref{pf:estimate-n1} now justified, average over $k=1,\ldots,n$ to get
\begin{align*}
\E\left[\int_{\C^d}\|x\|_t^\gamma\widehat{\mu}^{n,x}(dx)\right] &\le C\E\left\{1 + \int_0^t\left[\int_{\C^d}\|x\|_s^\gamma \widehat{\mu}^{n,x}(dx) + \frac{1}{n}\sum_{k=1}^n\int_A|a|^\gamma\Lambda^k_s(da)\right]ds \right\}.
\end{align*}
The second claim \eqref{def:estimate-n-2} now follows from Gronwall's inequality. The first claim \eqref{def:estimate-n-1} follows from \eqref{pf:estimate-n1}, \eqref{def:estimate-n-2}, and Gronwall's inequality.
\end{proof}

\begin{lemma} \label{le:estimate-optimal-n}
Let $\epsilon \ge 0$.
There exists a constant $C \ge 0$ such that, for all $n$ and all weak $\epsilon$-optimal controls $P \in \RC_n$, we have
\begin{align*}
\frac{1}{n}\sum_{k=1}^n\E^P\int_0^T\int_A|a|^{p'}\Lambda^k_t(da)dt \le C.
\end{align*}
\end{lemma}
\begin{proof}
Fix $P \in \RC_n$, and recall that $\E^P\int_0^T\int_A|a|^p\Lambda^k_t(da)dt < \infty$ for all $k$, as this ensures that the following expressions are well-defined.
Use the upper bounds on $f$ and $g$ from assumption (A.5) along with Lemma \ref{le:estimate-n} to get
\begin{align}
\E^P[\Gamma(\widehat{\mu}^n)] &\le C\E^P\left[1 + \int_{\C^d}\|x\|^p\widehat{\mu}^{n,x}(dx)\right] - c_3\frac{1}{n}\sum_{k=1}^n\E^P\int_0^T\int_A|a|^{p'}\Lambda^k_t(da)dt \nonumber \\
	&\le C\left(1 + \frac{1}{n}\sum_{k=1}^n\E^P\int_0^T\int_A|a|^p\Lambda^k_t(da)dt\right) - c_3\frac{1}{n}\sum_{k=1}^n\E^P\int_0^T\int_A|a|^{p'}\Lambda^k_t(da)dt. \label{pf:control-estimate-n1}
\end{align}
As usual, $C > 0$ is a constant, independent of $n$ and $P$, which can change from line to line.
On the other hand, fix arbitrarily a constant control $a_0 \in A$, and let $P^0_n$ denote an element of $\RC_n$ satisfying $P^0_n(\Lambda^k_t = \delta_{a_0}, \ a.e. \ t)=1$ for all $k=1,\ldots,n$ (as in Remark \ref{re:rcn-nonempty}, the existence of such $P^0_n$ follows from the results of Stroock and Varadhan \cite{stroockvaradhanbook}). Lemma \ref{le:estimate-n} implies
\[
\E^{P^0_n}\left[\int_{\C^d}\|x\|^{p'}\widehat{\mu}^{n,x}(dx)\right] = \frac{1}{n}\sum_{k=1}^n\E^{P^0_n}\left[\|X^k\|^{p'}\right] \le C(1 + T|a_0|^{p'}).
\]
Then use the lower bounds of Assumption (A.5) to show $\inf_n \E^{P^0_n}[\Gamma(\widehat{\mu}^n)] > -\infty$.
Now, if $P \in \RC_n$ is $\epsilon$-optimal, we have $\E^P[\Gamma(\widehat{\mu}^n)] \ge \E^{P^0_n}[\Gamma(\widehat{\mu}^n)] - \epsilon$, which combined with \eqref{pf:control-estimate-n1} yields a constant $C \ge 0$ such that, for all $n$ and all $\epsilon$-optimal $P \in \RC_n$,
\begin{align*}
\frac{1}{n}\sum_{k=1}^n\E^P\int_0^T\int_A\left(|a|^{p'}- C|a|^p\right)\Lambda^k_t(da)dt \le C.
\end{align*}
\end{proof}

\subsection{A tightness criterion}
Finally, we state without proof a useful tailor-made compactness result for controlled It\^o processes, which is essentially an application of Aldous' criterion. Its first claim is proven in \cite[Proposition B.4]{lacker-mfgcontrolledmartingaleproblems} and its second in \cite[Proposition 5.3]{lacker-mfglimit}, or rather an easy extension thereof (as it did not allow for control in the volatility).

\begin{proposition} \label{pr:itocompactness}
Fix $c > 0$. For $\kappa > 0$, let $\Q_\kappa$ denote the set of laws $\PP \circ (X,\Lambda)^{-1}$, where
\begin{enumerate}
\item $(\Omega,\F,\FF,\PP)$ is a filtered probability space supporting a $d$-dimensional $\FF$-adapted process $X$, a $\P^p(A)$-valued $\FF$-predictable process $\Lambda$, and a vector $N=(N^1,\ldots,N^{d_W})$ of orthogonal $\FF$-martingale measures, each with intensity $\Lambda_t(da)dt$.
\item The state equation holds,
\[
dX_t = \int_AB(t,a)\Lambda_t(da)dt + \int_A\Sigma(t,a)N(da,dt),
\]
where $(B,\Sigma) : [0,T] \times \Omega \times A \rightarrow \R^d \times \R^{d\times d_W}$ are jointly measurable with respect to the predictable $\sigma$-field of $[0,T]\times \Omega$ and the Borel $\sigma$-field of $A$.
\item It holds for all $(t,a)$, a.s., that
\[
|B(t,a)| \le c\left(1 + |X_t| + |a|\right), \quad\quad |\Sigma(t,a)|^2 \le c\left(1 + |X_t|^{p_\sigma} + |a|^{p_\sigma}\right).
\]
\item Lastly, we have
\[
\E\left[|X_0|^{p'} + \int_0^T\int_A|a|^{p'}\Lambda_t(da)dt\right] \le \kappa.
\]
\end{enumerate}
(That is, $\Q_\kappa$ is defined by varying the probability space as well as $B$ and $\Sigma$.)
Then $\Q_\kappa$ is precompact in $\P^p(\C^d\times\V)$. Moreover, if a triangular array $\{\kappa_{n,i} : 1 \le i \le n\} \subset [0,\infty)$ satisfies $\sup_n\frac{1}{n}\sum_{i=1}^n\kappa_{n,i} < \infty$, then the set
\[
\left\{\frac{1}{n}\sum_{i=1}^nQ_i : n \ge 1, \ i=1,\ldots,n, \ 	Q_i \in \Q_{\kappa_{n,i}}\right\}
\]
is precompact in $\P^p(\C^d\times\V)$.
\end{proposition}

\section{Proofs of existence Theorems \ref{th:existence} and \ref{th:markovian}} \label{se:existenceproofs}

The existence Theorem \ref{th:existence} is an immediate consequence of the following Lemmas \ref{le:gamma-usc} and \ref{le:rceps-compact}, which reduces the problem to maximizing an upper semicontinuous function on a compact set. Throughout the section, assumption \ref{assumption:A} is in force. The first lemma is essentially contained in \cite[Lemma 4.5]{lacker-mfgcontrolledmartingaleproblems} and \cite[Lemma 4.5]{lacker-mfglimit}, but we include the proof for the sake of transparency.

\begin{lemma} \label{le:gamma-usc}
Under assumption \ref{assumption:A}, $\Gamma$ is upper semicontinuous on $\P^p(\C^d\times\V)$, and the map $\P^p(\P^p(\C^d\times\V)) \ni P \mapsto \E^P[\Gamma(\mu)]$ is upper semicontinuous. When assumption \ref{assumption:B} holds as well, the latter function is continuous when restricted to any set $K \subset \P^p(\P^p(\C^d\times\V))$ satisfying
\begin{align}
\lim_{r\rightarrow\infty}\sup_{P \in K}\E^P\left[\int_{\C^d\times\V} Z1_{\{Z \ge r\}}\,d\mu\right] = 0, \quad \text{ where } \quad Z(x,q)=\|x\|^{p'} + \int_0^T\int_A|a|^{p'}q_t(da)dt. \label{def:uniformlyintegrable}
\end{align}
\end{lemma}
\begin{proof}
Upper semicontinuity of the map $F$ defined on $\P^p(\C^d\times\V) \times \C^d\times\V$ by
\[
F(m,x,q) = \int_0^T\int_Af(t,x_t,m^x_t,a)q_t(da)dt + g(x_T,m^x_T)
\]
follows from upper semicontinuity of $f$ and $g$ (assumption (A.3)) and the growth assumption (A.5) (see \cite[Corollary A.5]{lacker-mfgcontrolledmartingaleproblems} for details). This is enough to conclude (e.g., using Skorohod representation and Fatou's lemma) that $\Gamma(m) = \int F(m,\cdot)\,dm$ is upper semicontinuous. To prove the second claimed upper semciontinuity, note that assumption (A.5) implies that there exists $C > 0$ (which can change from line to line) such that for all $m \in \P^p(\C^d\times\V)$ we have
\begin{align*}
\Gamma(m) &= \int F(m,\cdot)\,dm \le C\int_{\C^d}m^x(dx)\left(1 + \|x\|^p + \int_{\C^d}\|z\|^pm^x(dz)\right) \\
	&\le C\left(1 + \int_{\C^d}\|z\|^pm^x(dz)\right)  \\
	&= C\left(1 + \ell^p_{\C^d,p}(m^x,\delta_0)\right) \\
	&\le C\left(1 + \ell^p_{\C^d\times\V,p}(m,\tilde{m})\right),
\end{align*}
where $\tilde{m} = \delta_{0} \times \delta_{q^0}$ for an arbitrary choice of $q^0 \in \V$, and where $\ell_{\C^d\times\V,p}$ was defined in \eqref{def:productmetric}. This is enough to prove the second claim; indeed, for a general complete separable metric space $(E,d)$, the map $\P^p(E) \ni \mu \mapsto \int\varphi\,d\mu$ is upper semicontinuous if $\varphi$ is upper semicontinuous and there exists $c > 0$ such that $\varphi(x) \le c(1+d(x,x_0))$ for all $x \in E$, for some $x_0 \in E$.

Finally, we prove the claimed restricted continuity, under the additional assumption \ref{assumption:B} which says that $f$ and $g$ are jointly continuous. Let $P_n \rightarrow P$ in $\P^p(\P^p(\C^d\times\V))$, with $P_n,P \in K$. We show first that
\begin{align}
(\E^{P_n}-\E^P)\left[\int_{\C^d \times \V}\int_0^T\int_Af(t,x_t,\mu^x_t,a)q_t(da)dt\mu(dx,dq)\right] \rightarrow 0. \label{pf:gammausc1}
\end{align}
To this end, define a probability measure $Q_n$ on $[0,T] \times \R^d \times \P^p(\R^d) \times A$ by
\[
Q_n(B) = \frac{1}{T}\E^{P_n}\left[\int_{\C^d\times\V}\int_0^T\int_A1_{\{(t,x_t,\mu^x_t,a) \in B\}}q_t(da)dt\mu(dx,dq)\right],
\]
and define $Q$ similarly in terms of $P$. It is clear that $Q_n \rightarrow Q$ weakly, because $P_n \rightarrow P$ weakly. Thus $Q_n \circ f^{-1} \rightarrow Q \circ f^{-1}$ weakly, as probability measures on $\R$. It follows from the assumption on the set $K$ and on the growth assumption (A.5) that
\[
\lim_{r\rightarrow\infty}\sup_n\int_{\{|f| \ge r\}}|f|\,dQ_n = 0.
\]
Thus $\int f\,dQ_n \rightarrow \int f\,dQ$, which is precisely \eqref{pf:gammausc1}. A similar argument shows
\[
(\E^{P_n}-\E^P)\left[\int_{\C^d}g(x,\mu)\mu^x(dx)\right] \rightarrow 0.
\]
Combining this and \eqref{pf:gammausc1} shows $\E^{P_n}[\Gamma(\mu)] \rightarrow \E^P[\Gamma(\mu)]$.
\end{proof}

The proof of the following compactness lemma makes some use of the following estimate, which follows immediately from assumption (A.4) and the fact that $1 \vee p_\sigma \le p$. Recall the definition of the generator $\L$ from \eqref{def:generator}. For every smooth compactly supported $\varphi$ on $\R^d$, there exists a constant $C>0$ depending only on $\varphi$ and the constant $c_1$ of assumption \ref{assumption:A} such that
\begin{align}
|\L\varphi(t,x,m,a)| \le C\left(1 + |x|^p + \int_{\R^d}|z|^pm(dz) + |a|^p\right), \label{def:generatorbound}
\end{align}
for all $(t,x,m,a) \in [0,T] \times \R^d \times \P^p(\R^d) \times A$.

\begin{lemma} \label{le:rceps-compact}
Given $\epsilon \ge 0$, let $\RC^\epsilon \subset \RC$ be the of $\epsilon$-optimal controls, as in Lemma \ref{le:estimate-optimal}. Then $\RC^\epsilon$ is compact in $\P^p(\C^d\times\V)$.
\end{lemma}
\begin{proof}
Lemma \ref{le:estimate-optimal} says that

\[
\sup_{m \in \RC^\epsilon}\E^m\int_0^T\int_A|a|^{p'}\Lambda_t(da)dt < \infty.
\]
According to Lemma \ref{le:estimate-MF}, this impies
\[
\sup_{m \in \RC^\epsilon}\int_{\C^d}\|x\|^{p'}m^x(dx) = \sup_{m \in \RC^\epsilon}\E^m[\|X\|^{p'}] < \infty.
\]
It now follows easy from the first claim of Proposition \ref{pr:itocompactness} that $\RC^\epsilon$ is precompact in $\P^p(\C^d\times\V)$.
To show that $\RC^\epsilon$ is closed, note that
\[
\RC^\epsilon = \{m \in \RC : \Gamma(m) \ge V - \epsilon\}, \quad\quad \text{ where } \quad\quad V = \sup_{m \in \RC}\Gamma(m).
\]
Now let $m^n \rightarrow m^\infty$ in $\P^p(\C^d\times\V)$, with $m^n \in \RC^\epsilon$. Upper semicontinuity of $\Gamma$ (see Lemma \ref{le:gamma-usc}) implies
\[
\Gamma(m^\infty) \ge \limsup_{n\rightarrow\infty}\Gamma(m^n) \ge V - \epsilon,
\]
and it remains only to show that $m^\infty$ belongs to $\RC$.
Since $X_0$ has law $\lambda$ under $m^n$, the same is true under $m^\infty$. For a smooth compactly supported function $\varphi$ and for $m \in \P^p(\C^d\times\V)$, define $M^{m,\varphi}_t : \C^d\times\V \rightarrow \R$ by
\begin{align}
M^{m,\varphi}_t(x,q) = \varphi(x_t) - \int_0^t\int_A\L\varphi(s,x_s,m^x_t,a)q_s(da)ds. \label{def:M-martingale}
\end{align}
The estimate \eqref{def:generatorbound} yields
\begin{align}
|M^{m,\varphi}_t(x,q)| \le C\left(1 + \|x\|^p  + \int_{\C^d}\|z\|^pm^x(dz) + \int_0^T\int_A|a|^pq_t(da)dt\right). \label{def:M-martingale-estimate}
\end{align}
Using this and the continuity of $(b,\sigma)$, it is readily checked that $(m,x,q) \mapsto M^{m,\varphi}_t(x,q)$ is a continuous function for each $t$ and $\varphi$, e.g., using \cite[Corollary A.5]{lacker-mfgcontrolledmartingaleproblems}. 
Since $m^n \rightarrow m^\infty$ in $\P^p(\C^d\times\V)$, it follows that
\begin{align*}
\E^{m^\infty}[(M^{m^\infty,\varphi}_t-M^{m^\infty,\varphi}_s)h] 
	&= \lim_n\E^{m^n}[(M^{m^n,\varphi}_t-M^{m^n,\varphi}_s)h],
\end{align*}
for every smooth compactly supported $\varphi$ and every bounded continuous  function $h$ on $\C^d\times\V$ which is measurable with respect to $\sigma(X_s, \Lambda_s : s \le t)$. Because $m^n$ is in $\RC$, the process $(M^{m^n,\varphi}_t(X,\Lambda))_{t \in [0,T]}$ is a martingale under $m^n$, and the above quantity is zero. This shows that $(M^{m^\infty,\varphi}_t(X,\Lambda))_{t \in [0,T]}$ is a martingale under $m^\infty$, and so $m^\infty \in \RC$; see Appendix \ref{ap:separating} for a short explanation of why it suffices here to consider only bounded continuous $h$.
\end{proof}

\subsection*{Proof of Theorem \ref{th:existence}}
Fix $\epsilon > 0$, and note that $\sup_{m \in \RC}\Gamma(m) = \sup_{m \in \RC^\epsilon}\Gamma(m)$. By Lemma \ref{le:rceps-compact}, $\RC^\epsilon$ is compact, and by Lemma \ref{le:gamma-usc}, $\Gamma$ is upper semicontinuous. Therefore, the supremum is attained. \hfill\qedsymbol

\subsection*{Proof of Theorem \ref{th:markovian}}

As in \cite[Theorem 2.5(a)]{elkaroui-compactification}, there exists a measurable function $\tilde{\sigma} : [0,T] \times \R^d \times \P^p(\R^d) \times \P^p(A) \rightarrow \R^{d\times d}$ such that 
\[
\tilde{\sigma}\tilde{\sigma}^\top(t,x,m,q) = \sigma\sigma^\top(t,x,m,a)q(da), \text{ for all } (t,x,m,q),
\]
and also $\tilde{\sigma}(t,x,m,\delta_a) = \sigma(t,x,m,a)$ for $a \in A$. Moreover, given $m \in \RC$, we may find a filtered probability space $(\Omega^1,\F^1,\FF^1,\PP^1)$ supporting a $d$-dimensional $\FF^1$-Wiener process $W^1$, a $d$-dimensional $\FF^1$-adapted process $X^1$, and a $\P^p(A)$-valued $\FF^1$-predictable process $\Lambda^1$ such that $\PP \circ (X^1,\Lambda^1)^{-1} = m$ and
\[
dX^1_t = \int_Ab(t,X^1_t,m^x_t,a)\Lambda^1_t(da)dt + \tilde{\sigma}(t,X^1_t,m^x_t,\Lambda^1_t)dW^1_t.
\]
The convexity assumption \ref{assumption:C} entails that, almost surely,
\begin{align}
\int_A(b,\sigma\sigma^\top,f)(t,X^1_t,m^x_t,a)\Lambda_t(da) \in K(t,X^1_t,m^x_t). \label{pf:markovian-sigmatilde}
\end{align}
By \cite[Proposition 3.5]{haussmannlepeltier-existence}, $K(t,x,m^x_t)$ is a closed set for each $(t,x)$. The measurable selection result of \cite[Theorem A.9]{haussmannlepeltier-existence} (or rather an extension in \cite[Lemma 3.1]{dufourstockbridge-existence}) implies that there exist measurable functions $\hat{\alpha} : [0,T] \times \R^d \rightarrow A$  and $\hat{z} : [0,T] \times \R^d \rightarrow [0,\infty)$ such that 
\begin{align}
\E\left[\left. \int_Ab(t,X^1_t,m^x_t,a)\Lambda^1_t(da) \right| X^1_t\right] &= b(t,X^1_t,m^x_t,\hat{\alpha}(t,X^1_t)), \label{pf:markovian1} \\
\E\left[\left. \int_A\sigma\sigma^\top(t,X^1_t,m^x_t,a)\Lambda^1_t(da) \right| X^1_t\right] &= \sigma\sigma^\top(t,X^1_t,m^x_t,\hat{\alpha}(t,X^1_t)), \label{pf:markovian1-2} \\
\E\left[\left. \int_Af(t,X^1_t,m^x_t,a)\Lambda^1_t(da) \right| X^1_t\right] &= f(t,X^1_t,m^x_t,\hat{\alpha}(t,X^1_t)) - \hat{z}(t,X^1_t). \label{pf:markovian2}
\end{align}
Note that in \eqref{pf:markovian1-2} and \eqref{pf:markovian-sigmatilde} together imply
\begin{align*}
\E\left[\left. \int_A\tilde{\sigma}\tilde{\sigma}^\top(t,X^1_t,m^x_t,a)\Lambda^1_t(da) \right| X_t\right] &= \sigma\sigma^\top(t,X^1_t,m^x_t,\hat{\alpha}(t,X^1_t)).
\end{align*}
Thanks to \eqref{pf:markovian1}, the mimicking theorem of Brunick and Shreve \cite{brunickshreve-mimicking} (a generalization of a well known result of Gy\"ongy \cite{gyongy-mimicking}) then implies that there exists a filtered probability space $(\Omega^2,\F^2,\FF^2,\PP^2)$ supporting a $d_W$-dimensional $\FF^2$-Wiener process $W^2$ and a $d$-dimensional $\FF^2$-adapted process $X^2$ such that
\[
dX^2_t = b(t,X^2_t,m^x_t,\hat{\alpha}(t,X^2_t))dt + \sigma(t,X^2_t,m^x_t,\hat{\alpha}(t,X^2_T))dW^2_t,
\]
and also $\PP^2 \circ (X^2_t)^{-1} = \PP^1 \circ (X^1_t)^{-1} = m^x_t$ for each $t \in [0,T]$. Define a $\P^p(A)$-valued process $\Lambda^2$ by $\Lambda^2_t = \delta_{\hat{\alpha}(t,X^2_t)}$, and let $\widetilde{m} = \PP^2 \circ (X^2,\Lambda^2)^{-1}$.
Then $\widetilde{m}$ belongs to $\RC$ and is Markovian, and also $\widetilde{m}^x_t=m^x_t$ for all $t\in [0,T]$. Finally, use Fubini's theorem and  \eqref{pf:markovian2} to get, since $\hat{z} \ge 0$,
\begin{align*}
\Gamma(m) &= \E^{\PP^1}\left[\int_0^T\int_Af(t,X^1_t,m^x_t,a)\Lambda^1_t(da)dt + g(X^1_T,m^x_T)\right] \\
	&= \E^{\PP^1}\left[\int_0^T\left(f(t,X^1_t,m^x_t,\hat{\alpha}(t,X^1_t)) - \hat{z}(t,X^1_t)\right)dt + g(X^1_T,m^x_T)\right] \\
	&= \E^{\PP^2}\left[\int_0^T\left(f(t,X^2_t,\widetilde{m}^x_t,\hat{\alpha}(t,X^2_t)) - \hat{z}(t,X^2_t)\right)dt + g(X^2_T,\widetilde{m}^x_T)\right] \\
	&\le \E^{\PP^2}\left[\int_0^Tf(t,X^2_t,\widetilde{m}^x_t,\hat{\alpha}(t,X^2_t))dt + g(X^2_T,\widetilde{m}^x_T)\right] \\
	&= \Gamma(\widetilde{m}).
\end{align*}

\section{Limits of $n$-state controls} \label{se:limits}

The proofs of both Theorems \ref{th:main-limit} and \ref{th:converse-limit} involve  similar constructions, detailed in the two propositions of this section. In fact, these two key results comprise the bulk of the proofs, by identifying limit points of various sequences of $n$-state controls. The first proposition proves all of Theorem \ref{th:main-limit} except for the claimed optimality of the limit points, and the second shows that every candidate control in the McKean-Vlasov control problem can be realized as a limit of $n$-state controls.

\begin{proposition} \label{pr:tight-limits}
Suppose $P_n \in \RC_n$ satisfy
\begin{align}
\sup_n\frac{1}{n}\sum_{k=1}^n\E^{P_n}\int_0^T\int_A|a|^{p'}\Lambda^k_t(da)dt < \infty. \label{def:tightness-assumption}
\end{align}
Then $(P_n \circ (\widehat{\mu}^n)^{-1})$ is precompact in $\P^p(\P^p(\C^d\times\V))$, and every limit point is supported on $\RC$.
\end{proposition}

\begin{proposition} \label{pr:propagation}
Let $m \in \RC$. Then there exists $P_n \in \RC_n$ such that $P_n \circ (\widehat{\mu}^n)^{-1} \rightarrow \delta_m$ in $\P^p(\P^p(\C^d\times\V))$ and $\E^{P_n}[\Gamma(\widehat{\mu}^n)] \rightarrow \Gamma(m)$.
\end{proposition}

The proofs make some use of the metric $d_\V$ on $\V$ defined in \eqref{def:d_V}.
Fix arbitrarily some $a_0 \in A$, and let $q^0(dt,da) = dt\delta_{a_0}(da)$. Then any $q \in \V$ can be coupled with $q^0$ via the measure $\pi$ on $[0,T]^2 \times A^2$ given by $\pi(dt,dt',da,da') = dt\delta_t(dt')q_t(da)\delta_{a_0}(da)$, and this gives rise to the estimate
\begin{align}
d_\V(q,q^0) \le \left(\frac{1}{T}\int_0^T\int_A|a-a_0|^pq_t(da)dt\right)^{1/p} \le |a_0| + \left(\frac{1}{T}\int_0^T\int_A|a|^pq_t(da)dt\right)^{1/p}. \label{def:dV-estimate}
\end{align}
Recall that $\C^d \times \V$ is equipped with the metric $d_{\C^d \times \V}$ defined in \eqref{def:productmetric}.

\subsection*{Proof of Proposition \ref{pr:tight-limits}}
We adapt to the controlled setting a martingale argument which is by now classical in  McKean-Vlasov limit theory (c.f. \cite{oelschlagermkv,gartnermkv} for uncontrolled and \cite{budhirajadupuisfischer} for controlled diffusions).
Throughout the proof, we will make use of the notation $\langle m,\varphi\rangle$ in place of $\int\varphi\,dm$.
Let $q^0 \in \V$ be defined as above.
According to \cite[Corollary B.2]{lacker-mfgcontrolledmartingaleproblems}, to prove precompactness it suffices to check that 
\begin{align}
\sup_n\E^{P_n}\left[\int_{\C^d\times\V}d_{\C^d\times\V}((x,q),(0,q^0))^{p'}\widehat{\mu}^n(dq,dx)\right] < \infty, \label{pf:tightness-bound}
\end{align}
and also that the mean measures $(\E^{P_n}[\widehat{\mu}^n])$ are tight. The mean measures are defined by, for bounded measurable functions $\varphi$ on $\C^d\times\V$, 
\[
\langle \E^{P_n}[\widehat{\mu}^n],\varphi\rangle = \E^{P_n}\left[\langle\widehat{\mu}^n,\varphi\rangle\right] = \frac{1}{n}\sum_{k=1}^n\E^{P_n}\left[\varphi(X^k,\Lambda^k)\right].
\]
To prove \eqref{pf:tightness-bound}, it suffices in light of \eqref{def:dV-estimate} to show that
\begin{align}
\sup_n\frac{1}{n}\sum_{k=1}^n\E^{P_n}\left[\|X^k\|^{p'} + \int_0^T\int_A|a|^{p'}\Lambda^k_t(da)dt\right] < \infty. \label{pf:tightness-bound1}
\end{align}
But this follows from the assumption \eqref{def:tightness-assumption} and Lemma \ref{le:estimate-n}. Finally, to show that the mean measures are tight, simply use the second assertion of Proposition \ref{pr:itocompactness}.

The next task is to identify the limit points. Fix a limit point $P \in \P^p(\P^p(\C^d\times\V))$, and relabel the subsequence so that $P_n \rightarrow P$. First, note that $\widehat{\mu}^{n,x}_0 = \frac{1}{n}\sum_{i=1}^n\delta_{X^i_0}$ converges weakly to $\delta_{\lambda}$, since $(X_0^i)$ are i.i.d. with law $\lambda$ by assumption. That is, $P(\mu^x_0=\lambda)=1$, where $\mu$ denotes the identity map on $\P^p(\C^d\times\V)$. To prove that $P(\mu \in \RC)=1$, it remains to show that the martingale problem is satisfied at the limit. That is, defining $M^{m,\varphi}_t(x,q)$ as in \eqref{def:M-martingale}, we must show that
\begin{align}
P\left((M^{\mu,\varphi}_t)_{t \in [0,T]} \text{ is a martingale under } \mu, \ \forall \varphi\right) = 1, \label{pf:mtgproblem-limit}
\end{align}
where ``$\forall \varphi$'' means ``for all smooth compactly supported functions $\varphi$.''
To this end, recall the useful estimate \eqref{def:M-martingale-estimate} as well as the discussion thereafter, namely that $M_t^{m,\varphi}(x,q)$  is jointly continuous in $(m,x,q) \in \P^p(\C^d \times \V) \times \C^d \times \V$ for each fixed $t$ and $\varphi$.

Now, use the SDE representation of Proposition \ref{pr:SDErepresentation-n} along with It\^o's formula to see that $M^{\widehat{\mu}^n,\varphi}_t(X^k,\Lambda^k)$ is a martingale under $P_n$ with quadratic variation
\begin{align*}
\int_0^t\int_A\left|\sigma(s,X^k_s,\widehat{\mu}^{n,x}_s,a)\nabla\varphi(X^k_s)\right|^2
\Lambda_s(da)ds.
\end{align*}
In fact, for $k=1,\ldots,n$, these martingales $M^{\widehat{\mu}^n,\varphi}_t(X^k,\Lambda^k)$ are orthogonal. 
Fix $s < t$, and let $h : \C^d \times \V \rightarrow \R$ be bounded, continuous, and $\sigma(X_s,\Lambda_s : s \le t)$-measurable. We will show that
\begin{align}
\E^P\left[\langle\mu,h(M^{\mu,\varphi}_t-M^{\mu,\varphi}_s)\rangle^2\right] = 0. \label{pf:martingaleproblem1}
\end{align}
First note that
\begin{align*}
\E^{P_n}&\left[\langle\widehat{\mu}^n,h(M^{\widehat{\mu}^n,\varphi}_t-M^{\widehat{\mu}^n,\varphi}_s)\rangle^2\right] \\
	&= \E^{P_n}\left[\left(\frac{1}{n}\sum_{k=1}^nh(X^k,\Lambda^k)\left(M^{\widehat{\mu}^n,\varphi}_t(X^k,\Lambda^k) - M^{\widehat{\mu}^n,\varphi}_s(X^k,\Lambda^k)\right)\right)^2\right] \\
	&= \frac{1}{n^2}\sum_{k=1}^n\E^{P_n}\left[h(X^k,\Lambda^k)^2\left(M^{\widehat{\mu}^n,\varphi}_t(X^k,\Lambda^k) - M^{\widehat{\mu}^n,\varphi}_s(X^k,\Lambda^k)\right)^2 \right] \\
	&\le \frac{1}{n^2}\sum_{k=1}^n\E^{P_n}\left[h(X^k,\Lambda^k)^2\int_s^t\int_A\left|\sigma(u,X^k_u,\widehat{\mu}^{n,x}_u,a)\nabla\varphi(X^k_u)\right|^2\Lambda_u(da)du\right].
\end{align*}
Assumption (A.4) and $p_\sigma \le p'$ imply that there exists $C > 0$ (independent of $n$) such that
\begin{align*}
\int_s^t\int_A&\left|\sigma(u,X^k_u,\widehat{\mu}^{n,x}_u,a)\nabla\varphi(X^k_u)\right|^2\Lambda_u(da)du \\
	&\le C\int_s^t\int_A\left(1 + |X^k_u|^{p_\sigma} + \left(\int_{\R^d}|z|^p\widehat{\mu}^{n,x}_u(dz)\right)^{p_\sigma/p} + |a|^{p_\sigma}\right)\Lambda_u(da)du \\
	&\le C\left(1 + \|X^k\|^{p'} + \int_{\C^d}\|z\|^{p'}\widehat{\mu}^{n,x}(dz) + \int_0^T\int_A|a|^{p'}\Lambda^k_u(da)du\right).
\end{align*}
Along with \eqref{pf:tightness-bound1}, this implies
\[
\sup_nn\E^{P_n}\left[\langle\widehat{\mu}^n,h(M^{\widehat{\mu}^n,\varphi}_t-M^{\widehat{\mu}^n,\varphi}_s)\rangle^2\right] < \infty.
\]
This in turn implies
\begin{align*}
\E^P\left[\langle\mu,h(M^{\mu,\varphi}_t-M^{\mu,\varphi}_s)\rangle^2\right] &\le \liminf_{n\rightarrow\infty}\E^{P_n}\left[\langle\widehat{\mu}^n,h(M^{\widehat{\mu}^n,\varphi}_t-M^{\widehat{\mu}^n,\varphi}_s)\rangle^2\right] =0,
\end{align*}
because the map $\P^p(\C^d \times \V) \ni m \mapsto \langle m,h(M^{m,\varphi}_t-M^{m,\varphi}_s)\rangle^2$ is continuous and bounded from below; indeed, this follows from the aforementioned joint continuity of $M^{m,\varphi}_t(x,q)$ in $(m,x,q)$.

We have now proven \eqref{pf:martingaleproblem1}, and it follows that $\langle \mu,h(M^{\mu,\varphi}_t-M^{\mu,\varphi}_s)\rangle = 0$ holds $P$-almost surely, for each $s < t$, each smooth compactly supported $\varphi$, and each bounded continuous $\F_s$-measurable $h$. By applying this to a suitably dense countable set of $(s,t,\varphi,h)$ we can interchange the order of the quantifiers and conclude that $\langle \mu,h(M^{\mu,\varphi}_t-M^{\mu,\varphi}_s)\rangle = 0$ for each $s < t$ and each $\varphi$, $P$-almost surely. This proves \eqref{pf:mtgproblem-limit} and thus $P(\mu \in \RC) = 1$. To elaborate on this last point, it is clear that we can restrict our attention to $s$ and $t$ belonging to a dense subset of $[0,T]$ and to $\varphi$ belong to a dense set of smooth functions, whereas the separability of the class of functions $h$ is less immediate. See Appendix \ref{ap:separating} for details.
\hfill\qedsymbol

\subsection*{Proof of Proposition \ref{pr:propagation}}

The line of argument is often known as \emph{trajectorial propagation of chaos} (see \cite{sznitman}), constructing an explicit coupling between the limit and pre-limit state processes. To begin the proof, apply Proposition \ref{pr:SDErepresentation} to $m$, and then construct a sequence of independent copies. As a result we may find a filtered probability space $(\Omega,\F,\FF,\PP)$ supporting i.i.d. random variables $(X^i,\Lambda^i,N^i)$, satisfying the following:
\begin{enumerate}
\item Each $X^i$ is a $d$-dimensional $\FF$-adapted process.
\item Each $\Lambda^i$ is an $\FF$-predictable $\P^p(A)$-valued process.
\item Each $N^i=(N^{i,1},\ldots,N^{i,d_W})$ is a vector of orthogonal $\FF$-martingale measures, each with intensity measure $\Lambda^i_t(da)dt$.
\item The McKean-Vlasov equation holds for each $i$:
\[
dX^i_t = \int_Ab(t,X^i_t,m^x_t,a)\Lambda^i_t(da)dt + \int_A\sigma(t,X^i_t,m^x_t,a)N^i(da,dt).
\]
\item The law of $(X^i,\Lambda^i)$ is precisely $m$, for each $i$.
\end{enumerate}
In particular, Lemma \ref{le:estimate-optimal} combined with Lemma \ref{le:estimate-MF} together imply that, for each $i$,
\begin{align}
\E\left[\|X^i\|^{p'} + \int_0^T\int_A|a|^{p'}\Lambda^i_t(da)dt\right] < \infty. \label{pf:lipschitz1}
\end{align}
Thanks to the Lipschitz assumption \ref{assumption:B}, there exists a unique square-integrable (recall $p' \ge 2$) progressively measurable processes $(Z^{n,1},\ldots,Z^{n,n})$ such that
\begin{align*}
dZ^{n,i}_t &= \int_Ab(t,Z^{n,i}_t,\widetilde{\mu}^{n,x}_t,a)\Lambda^i_t(da)dt + \int_A\sigma(t,Z^{n,i}_t,\widetilde{\mu}^{n,x}_t,a)N^i(da,dt), \ Z^{n,i}_0=X^i_0, \\
\widetilde{\mu}^n &= \frac{1}{n}\sum_{k=1}^n\delta_{(Z^{n,k},\Lambda^k)}.
\end{align*}
(Indeed, existence and uniqueness here is an easy adaptation of now-standard arguments, which can be found in \cite{funaki1984certain,sznitman}.)
It follows from Proposition \ref{pr:SDErepresentation-n} that the law $P_n = \PP \circ ((Z^{n,i},\Lambda^i)_{i=1}^n)^{-1}$ belongs to $\RC_n$.
Define $\widehat{\mu}^n = \frac{1}{n}\sum_{i=1}^n\delta_{(X^i,\Lambda^i)}$. We claim that 
\begin{align}
\lim_{n\rightarrow\infty}\E[\|Z^{n,1} - X^1\|^{p'} + \ell_{\C^d\times\V,p'}^{p'}(\widehat{\mu}^n,\widetilde{\mu}^n)] \rightarrow 0. \label{pf:lipschitz2}
\end{align}
Assuming for the moment that \eqref{pf:lipschitz2} holds, we complete the proof as follows: Because $(X^i,\Lambda^i)$ are i.i.d. with law $m$, the law of large numbers implies $\widehat{\mu}^n \rightarrow m$ almost surely in $\P(\C^d \times \V)$, and the finite moments of \eqref{pf:lipschitz1} allow us upgrade this convergence to $\P^p(\C^d\times\V)$. Hence, \eqref{pf:lipschitz2} implies $\PP \circ (\widetilde{\mu}^n)^{-1} \rightarrow \delta_m$ in $\P^p(\P^p(\C^d\times\V))$. To conclude that $\E[\Gamma(\widetilde{\mu}^n)] \rightarrow \Gamma(m)$, simply note that \eqref{pf:lipschitz2} and \eqref{pf:lipschitz1}, together with exchangeability of $(Z^{n,i},\Lambda^i)_{i=1}^n$, verify the uniform integrability hypothesis of Lemma \ref{le:gamma-usc}.

The rest of the proof is devoted to justifying \eqref{pf:lipschitz2}.
For $k=1,\ldots,n$, use the Burkholder-Davis-Gundy inequality and the Lipschitz assumption (A.4) (noting that $\ell_{\R^d,p} \le \ell_{\R^d,p'}$) to find a constant $C$ (which will change from line to line) such that
\begin{align*}
\E[\|Z^{n,k}-X^k\|_t^{p'}] &\le C\E\left[ \left(\int_0^t\int_A|b(s,Z^{n,k}_s,\widetilde{\mu}^{n,x}_s,a) - b(s,X^k_s,m^x_s,a)|\Lambda^k_s(da)ds\right)^{p'} \right. \\
	&\quad\quad\quad\quad \left. + \left(\int_0^t\int_A|\sigma(s,Z^{n,k}_s,\widetilde{\mu}^{n,x}_s,a) - \sigma(s,X^k_s,m^x_s,a)|^2\Lambda^k_s(da)ds\right)^{p'/2}\right] \\
	&\le C\E\int_0^t\left[|Z^{n,k}_s-X^k_s|^{p'} + \ell_{\R^d,p'}^{p'}(\widetilde{\mu}^{n,x}_s,m^x_s)\right]ds.
\end{align*}
Use Gronwall's inequality to get
\begin{align}
\E[\|Z^{n,k}-X^k\|_t^{p'}] &\le C\E\int_0^t\ell_{\R^d,p'}^{p'}(\widetilde{\mu}^{n,x}_s,m^x_s)ds. \label{pf:propagation2}
\end{align}
Define the truncated Wasserstein distance $\ell_{s,\C^d,p'}$ by
\begin{align}
\ell_{s,\C^d,p'}^{p'}(m,m') = \inf\left\{\int_{\C^d\times\C^d} \|x-y\|_s^{p'}\pi(dx,dy) : \pi \in \P(\C^d \times \C^d) \text{ has marginals } m, \ m'\right\}. \label{def:truncated-wasserstein}
\end{align}
It is straightforward to check that for every $m^1,m^2 \in \P(\C^d)$ and every $s \in [0,T]$ we have $\ell_{\R^d,p'}(m^1_s,m^2_s) \le \ell_{s,\C^d,p'}(m^1,m^2)$.
Returning to \eqref{pf:propagation2}, use the obvious coupling and the triangle inequality to get
\begin{align*}
\E\left[\ell_{t,\C^d,p'}^{p'}(\widetilde{\mu}^{n,x},\widehat{\mu}^{n,x})\right] &\le \frac{1}{n}\sum_{k=1}^n\widetilde{\E}[\|Z^{n,k}-X^k\|_t^{p'}] \\
	&\le C\E\int_0^t\left(\ell_{s,\C^d,p'}^{p'}(\widetilde{\mu}^{n,x},\widehat{\mu}^{n,x}) + \ell_{s,\C^d,p'}^{p'}(\widehat{\mu}^{n,x},m^x)\right)ds.
\end{align*}
Another application of Gronwall's inequality yields, for all $t \in [0,T]$,
\begin{align*}
\E\left[\ell_{t,\C^d,p'}^{p'}(\widetilde{\mu}^{n,x},\widehat{\mu}^{n,x})\right] &\le C\E\int_0^t\ell_{s,\C^d,p'}^{p'}(\widehat{\mu}^{n,x},m^x)ds \le CT\E\left[\ell_{\C^d,p'}^{p'}(\widehat{\mu}^{n,x},m^x)\right]
\end{align*}
The finite moment \eqref{pf:lipschitz1} and the almost sure weak convergence $\widehat{\mu}^{n,x} \rightarrow m^x$ together imply that the above expectation tends to zero. Recalling the definition \eqref{def:productmetric} of the metric on $\C^d\times\V$, we have
\[
d_{\C^d\times\V}^{p'}((x,q),(x',q')) \le 2^{p'-1}(d_{\C^d}^{p'}(x,x') + d_{\V}^{p'}(q,q')).
\]
Since $\widetilde{\mu}^{n}$ and $\widehat{\mu}^{n}$ have the same $\V$-marginal, it follows that
\[
\E\left[\ell_{\C^d\times\V,p'}^{p'}(\widetilde{\mu}^{n},\widehat{\mu}^{n})\right] \le 2^{p'-1}\E\left[\ell_{\C^d,p'}^{p'}(\widetilde{\mu}^{n,x},\widehat{\mu}^{n,x})\right] \rightarrow 0.
\]
Recalling also \eqref{pf:propagation2}, this completes the proof of \eqref{pf:lipschitz2}. 
\hfill\qedsymbol

\section{Proofs of the limit theorems} \label{se:limittheorems-proof}

\subsection*{Proof of Theorem \ref{th:main-limit}}
Given the preparations of the previous section, the proof of the main limit theorem is now straightforward. Note that the second claim \eqref{def:distance-limit} follows immediately from the first by an application of the Portmanteau theorem to the closed set $\{m \in \P^p(\C^d\times\V) : \ell_{\C^d\times\V}(m,\RC^*) \ge \epsilon\}$. Let $P_n \in \RC_n$ denote an $\epsilon_n$-optimal control for the $n$-state problem. Lemma \ref{le:estimate-optimal-n} implies 
\[
\sup_n\frac{1}{n}\sum_{k=1}^n\E^{P_n}\int_0^T\int_A|a|^{p'}\Lambda^k_t(da)dt < \infty.
\]
By Proposition \ref{pr:tight-limits}, $(P_n)$ is precompact in $\P^p(\P^p(\C^d\times\V))$, and every limit point is concentrated on $\RC$. Let $P$ denote a limit point, and relabel the subsequence so that $P_n \circ (\widehat{\mu}^n)^{-1} \rightarrow P$. Then $P(\mu \in \RC)=1$, and to prove $P(\mu \in \RC^*)=1$ it suffices to show that
\begin{align}
\E^P[\Gamma(\mu)] \ge \Gamma(m), \text{ for all } m \in \RC. \label{pf:optimality1}
\end{align}
First, use the upper semicontinuity of $\Gamma$ of Lemma \ref{le:gamma-usc} to get
\[
\E^P[\Gamma(\mu)]  \ge \limsup_{n\rightarrow\infty}\E^{P_n}[\Gamma(\widehat{\mu}^n)].
\]
Fix $m \in \RC$, and use Proposition \ref{pr:propagation} to find $Q_n \in \RC_n$ such that $\E^{Q_n}[\Gamma(\widehat{\mu}^n)] \rightarrow \Gamma(m)$. The $\epsilon_n$-optimality of $P_n$ implies
\[
\E^{P_n}[\Gamma(\widehat{\mu}^n)] \ge \E^{Q_n}[\Gamma(\widehat{\mu}^n)] - \epsilon_n.
\]
Recalling that $\epsilon_n \rightarrow 0$, we have thus proven \eqref{pf:optimality1}. \hfill\qedsymbol

\subsection*{Proof of Theorem \ref{th:converse-limit}}
The second claim, that $P_n$ can be taken to be strong controls, follows from the first claim and from Theorem \ref{th:strongequalsweak-n}. Hence, we prove only the first claim.
Define $L \subset \P^p(\C^d\times\V)$ to be the set of $\lim_{n\rightarrow}P_n \circ (\widehat{\mu}^n)^{-1}$ such that $P_n$ is an $\epsilon_n$-optimal control for each $n$, for some sequence $\epsilon_n \rightarrow 0$. Let $L_s$ denote the set of \emph{subsequential} limits of such sequences. Naturally, write $\P(\RC^*) = \{M \in \P(\P^p(\C^d\times\V)) : M(\RC^*)=1\}$. Because $\RC^*$ is compact in $\P^p(\C^d\times\V)$ by Lemma \ref{le:rceps-compact},  it follows that $\P(\RC^*) \subset \P^p(\P^p(\C^d\times\V))$, that $\P^p(\P^p(\C^d\times\V))$ and $\P(\P^p(\C^d\times\V))$ induce the same topology on $\P(\RC^*)$, and finally that $\P(\RC^*)$ is compact with respect to either of these topologies. Henceforth, we work with the latter (weak convergence) topology on $\P(\RC^*)$.

First note that
\[
L \subset L_s \subset \P(\RC^*),
\]
where the first inclusion is obvious, and the second is the content of Theorem \ref{th:main-limit}. The proof will be complete if we show that $\P(\RC^*) \subset L$. Note that $\P(\RC^*)$ is a compact convex set, and the set of extreme points is $\{\delta_m : m \in \RC^*\}$. Hence, by the Krein-Milman theore, to show that $\P(\RC^*) \subset L$ it suffices to show that $L$ is closed and convex and that $\delta_m \in L$ for each $m \in \RC^*$.

\textbf{Step 1:}
We first show that $L$ is convex. Let $M^1,M^2 \in L$, so for $i=1,2$ we may find $\epsilon^i_n\rightarrow 0$ and $\epsilon^i_n$-optimal controls $P^i_n \in \RC_n$ such that $M^i = \lim_{n\rightarrow\infty}P^i_n \circ (\widehat{\mu}^n)^{-1}$. Let $t \in (0,1)$. Clearly $\epsilon_n = t\epsilon^1_n + (1-t)\epsilon^2_n$ tends to zero. Moreover, $P_n := tP^1_n + (1-t)P^2_n$ belongs to $\RC_n$, because $\RC_n$ is easily seen to be convex. Because the objective functional $\RC_n \ni \P \mapsto \E^P[\Gamma(\widehat{\mu}^n)]$ is affine, $P_n$ is a $\epsilon_n$-optimal control. This shows $tM^1+(1-t)M^2 = \lim_{n\rightarrow\infty}P_n$ belongs to $L$.

\textbf{Step 2:}
We next show that $L$ is closed. For $n \ge 1$ and $\epsilon \ge 0$ define
\[
S^\epsilon_n = \left\{P_n \circ (\widehat{\mu}^n)^{-1} : P_n \in \RC_n \text{ is an } \epsilon\text{-optimal control}\right\} \subset \P(\P^p(\C^d\times\V)).
\]
Note that $S^\delta_n \subset S_n^\epsilon$ for all $\delta < \epsilon$.
It is easy to see that $L$ is precisely the set of $M \in \P(\P^p(\C^d\times\V))$ such that, for every $\epsilon > 0$ and every open set $U$ containing $M$, there exists $N$  such that $U \cap S^\epsilon_n \neq \emptyset$ for all $n \ge N$. Now fix $M \notin L$. Find $\epsilon > 0$, an open set $U$ containing $M$, and $n_k \rightarrow \infty$ such that $U \cap S^{\epsilon}_{n_k} = \emptyset$ for all $k$. Then $U \cap \overline{S}^{\epsilon}_{n_k} = \emptyset$. Define $V = U \backslash \bigcap_{k=1}^\infty \overline{S}^\epsilon_{n_k}$, and notice that $V$ is open. Clearly $V$ contains $M$. By construction, $V \cap S^\epsilon_{n_k} = \emptyset$ for each $k$, which shows that $V \subset L^c$. Thus $L^c$ is open.

\textbf{Step 3:}
Finally, we show that $L$ contains the extreme points $\{\delta_m : m \in\RC^*\}$ of $\P(\RC^*)$. Fix $m^* \in \RC^*$. By Proposition \ref{pr:propagation}, there exist $P_n \in \RC_n$ such that $P_n \circ (\widehat{\mu}^n)^{-1} \rightarrow \delta_{m^*}$ and $\E^{P_n}[\Gamma(\widehat{\mu}^n)] \rightarrow \Gamma(m^*)$. Let $P_n^* \in \RC_n$ be an $n$-state optimal control for each $n$, the existence of which is guaranteed by Theorem \ref{th:n-player-existence}. Let
\[
\epsilon_n = \E^{P^*_n}[\Gamma(\widehat{\mu}^n)] - \E^{P_n}[\Gamma(\widehat{\mu}^n)].
\]
Optimality of $P_n^*$ ensures that $\epsilon_n \ge 0$.
By Theorem \ref{th:main-limit}, every limit point of $P_n^* \circ (\widehat{\mu}^n)^{-1}$ is supported on $\RC^*$. By upper semicontinuity of $\Gamma$ (see Lemma \ref{le:gamma-usc}), this implies 
\[
\limsup_{n\rightarrow\infty}\E^{P^*_n}[\Gamma(\widehat{\mu}^n)]  \le \sup_{m \in \RC^*}\Gamma(m) = \Gamma(m^*) = \lim_{n\rightarrow\infty}\E^{P_n}[\Gamma(\widehat{\mu}^n)].
\]
This shows $\epsilon_n \rightarrow 0$, completing the proof.
\hfill\qedsymbol

\section{Strong versus relaxed formulations} \label{se:strongvsweak}
This section is devoted to the proof of Theorem \ref{th:strongequalsweak-MF}. The proof of Theorem \ref{th:strongequalsweak-n} is nearly identical up to notational changes, so we omit it. Recall that we assume throughout that $p'=2$. We start with a lemma due mostly to M\'el\'eard \cite{meleard1992martingale}, modulo integrability issues.

\begin{lemma} \label{le:meleard}
Suppose $M$ is a martingale measure with intensity $\Lambda_t(da)dt$, defined on some filtered probability space supporting a $\P^p(A)$-valued process $\Lambda$. Suppose $\Lambda^n_t(da)dt \rightarrow \Lambda_t(da)dt$ weakly, almost surely, for some other $\P^p(A)$-valued processes $\Lambda^n$ satisfying
\begin{align}
\lim_{r\rightarrow\infty}\sup_n\E\left[\int_0^T\int_{\{|a| > r\}}|a|^2\Lambda^n_t(da)dt\right] < \infty. \label{def:mtgmeasureUI}
\end{align}
Then there exists a sequence of martingale measures $M^n$ (defined on an extension of the probability space) with intensities $\Lambda^n_t(da)dt$ such that
\begin{align}
\lim_{n\rightarrow\infty}\E\left[\left(\int_{A \times [0,T]}\varphi(t,a)M^n(da,dt)-\int_{A \times [0,T]}\varphi(t,a)M(da,dt)\right)^2\right] = 0, \label{def:mtgmeasure-approx}
\end{align}
for every predictable function $\varphi(t,a)$, continuous in $a$, and satisfying $|\varphi(t,a)|^2 \le c(Z + |a|^2)$ for all $a \in A$, for some integrable random variable $Z \ge 0$ and some $c > 0$.
\end{lemma}
\begin{proof}
The result of M\'el\'eard \cite{meleard1992martingale} provides a sequence of martingale measures $M^n$ with intensities $\Lambda^n_t(da)dt$ such that \eqref{def:mtgmeasure-approx} holds for all \emph{bounded} predictable $\varphi$ which are continuous in $a$. For general $\varphi$, let $r > 0$ and compute
\begin{align}
\E&\left[\left(\int\varphi(t,a)M^n(da,dt)-\int\varphi(t,a)M(da,dt)\right)^2\right] \nonumber \\
	&\le 3\E\left[\left(\int\varphi(t,a)1_{\{|\varphi| \le r\}}M^n(da,dt)-\int\varphi(t,a)1_{\{|\varphi| \le r\}}M(da,dt)\right)^2\right] \nonumber \\
	&\quad + 3\E\left[\left(\int\varphi(t,a)1_{\{|\varphi| > r\}}M^n(da,dt)\right)^2 + \left(\int\varphi(t,a)1_{\{|\varphi| > r\}}M(da,dt)\right)^2\right]. \label{pf:mtgmeasureapprox1}
\end{align}
The first term tends to zero, thanks to the aforementioned result of \cite{meleard1992martingale}. For the second term, notice that
\begin{align*}
\E\left[\left(\int\varphi(t,a)1_{\{|\varphi| > r\}}M^n(da,dt)\right)^2\right] &= \E\left[\int|\varphi(t,a)|^21_{\{|\varphi| > r\}}\Lambda^n_t(da)dt\right] \\
	&\le c\E\left[\int (Z+ |a|^2)1_{\{Z+|a|^2 > r/c\}}\Lambda^n_t(da)dt\right] \\
	&\le c\E\left[TZ1_{\{Z > r/2c\}} + \int |a|^2 1_{\{|a|^2 > r/2c\}}\Lambda^n_t(da)dt\right].
\end{align*}
Since $\E Z < \infty$ by assumption, this can be made arbitrarily small, uniformly in $n$, by sending $r\rightarrow\infty$. Similarly, by Fatou's lemma, the assumption \eqref{def:mtgmeasureUI} implies $\E\int_0^T\int_A|a|^2\Lambda_t(da)dt < \infty$,  and we see that choosing $r$ large can make the second term in line \eqref{pf:mtgmeasureapprox1} arbitrarily small.
\end{proof}

\subsection*{Proof of Theorem \ref{th:strongequalsweak-MF}}

Let $m \in \RC^*$, recalling that Theorem \ref{th:existence} ensures $\RC^* \neq \emptyset$. By Proposition \ref{pr:SDErepresentation}, there exists filtered probability space $(\Omega,\F,\FF,\PP)$ supporting $(X,\Lambda,M)$ satisfying the following:
\begin{enumerate}
\item $X$ is a $d$-dimensional $\FF$-adapted process.
\item $\Lambda$ is an $\FF$-predictable $\P^p(A)$-valued process.
\item $M=(M^1,\ldots,M^{d_W})$ are orthogonal $\FF$-martingale measures, each with intensity  $\Lambda_t(da)dt$.
\item The McKean-Vlasov equation holds:
\begin{align}
dX_t = \int_Ab(t,X_t,m^x_t,a)\Lambda_t(da)dt + \int_A\sigma(t,X^i_t,m^x_t,a)M(da,dt). \label{pf:valueequal0}
\end{align}
\item The law of $(X,\Lambda)$ is precisely $m$.
\end{enumerate}
By Lemma \ref{le:estimate-optimal}, we have (since $p'=2$)
\begin{align}
\E\left[\|X\|^2 + \int_0^T\int_A|a|^2\Lambda_t(da)dt\right] < \infty. \label{pf:valueequal1.1}
\end{align}
We do a three-step approximation:

\textbf{Step 1:}
We first approximate $\Lambda$ by bounded controls. For $n\ge 1$, let $\iota_n : A \rightarrow A$ be any measurable function such that $|\iota_n(a)| \le n$ for all $a \in A$ and $\iota_n(a)=a$ when $|a| \le n$. Define $\Lambda^n_t = \Lambda_t \circ \iota_n^{-1}$, so that clearly $\Lambda^n_t(da)dt \rightarrow \Lambda_t(da)dt$ a.s.
Note also that
\begin{align}
\lim_{r\rightarrow\infty}\sup_n\E\left[\int_0^T\int_{\{|a| > r\}}|a|^2\Lambda^n_t(da)dt\right] \le \lim_{r\rightarrow\infty}\left[\int_0^T\int_{\{|a| > r\}}|a|^2\Lambda_t(da)dt\right] = 0, \label{pf:valueequal-estimate1.0}
\end{align}
thanks to \eqref{pf:valueequal1.1} and the simple observation that $|\iota_n(a)| \le |a|$ for all $a \in A$.
By Lemma \ref{le:meleard}, there exists (on an enlargement of the probability space) a sequence of orthogonal martingale measures $M^n=(M^{n,1},\ldots,M^{n,d_W})$ such that $M^{n,i}$ has intensity measure $\Lambda^n_t(da)dt$ for each $i$ and
\[
\lim_{n\rightarrow\infty}\E\left[\left(\int_{A \times [0,T]}\varphi(t,a)M^n(da,dt)-\int_{A \times [0,T]}\varphi(t,a)M(da,dt)\right)^2\right] = 0,
\]
for every predictable function $\varphi$ (from $[0,T]\times\Omega\times A$ to $\R^{d\times d_W}$) satisfying $|\varphi(t,a)|^2 \le c(Z + |a|^2)$ for all $a \in A$, for some $c > 0$ and some integrable random variable $Z$.

Let $X^n$ denote the unique solution of the McKean-Vlasov equation
\begin{align*}
dX^n_t = \int_Ab(t,X^n_t,\PP \circ (X^n_t)^{-1},a)\Lambda^n_t(da)dt + \int_A\sigma(t,X^n_t,\PP \circ (X^n_t)^{-1},a)M^n(da,dt), \ \ X^n_0 = X_0.
\end{align*}
Note that the Lipschitz assumption \ref{assumption:B} ensures the well-posedness of this equation, by standard arguments. By Proposition \ref{pr:SDErepresentation} $\PP \circ (X^n,\Lambda^n)^{-1}$ belongs to $\RC$. The rest of this step is a long but straightforward proof, using the Lipschitz assumption, that $\E[\|X^n-X\|^2] \rightarrow 0$, from which the desired approximations will quickly follow.
Recall the notation $\|x\|_t = \sup_{s \in [0,t]}|x_s|$ and also the truncated Wasserstein distance $\ell_{t,\C^d,p}$ from \eqref{def:truncated-wasserstein}. Apply the Burkholder-Davis-Gundy inequality and Jensen's inequality to get (for a constant $C > 0$ which changes from line to line)
\begin{align*}
\E&\left[\|X^n-X\|_t^2\right] \\
	&\le C\E\left\{\int_0^t\int_A\left|b(s,X^n_s,\PP \circ (X^n_s)^{-1},a) - b(s,X_s,m^x_s,a)\right|^2\Lambda^n_s(da)ds\right. \\
	&\quad\quad + \left|\int_0^t\int_Ab(s,X_s,m^x_s,a)\Lambda^n_s(da)ds - \int_0^t\int_Ab(s,X_s,m^x_s,a)\Lambda_s(da)ds\right|^2 \\
	&\quad\quad + \left|\int_0^t\int_A\sigma(s,X^n_s,\PP \circ (X^n_s)^{-1},a) - \sigma(s,X_s,m^x_s,a)M^n(da,ds)\right|^2 \\
	&\quad\quad + \left.\left|\int_0^t\int_A\sigma(s,X_s,m^x_s,a)M^n(da,ds) - \int_0^t\int_A\sigma(s,X_s,m^x_s,a)M(da,ds)\right|^2\right\} \\
	&=: C\E[I_1+I_2+I_3+I_4]
\end{align*}
The Lipschitz assumption \ref{assumption:B} yields
\[
\E[I_1] \le C\int_0^t\left(\|X^n-X\|_s^2 + \ell_{s,\C^d,2}^2(\PP \circ (X^n)^{-1},m^x)\right)dt,
\]
and also
\begin{align*}
\E[I_3] &= \E\left[\int_0^t\int_A\left|\sigma(s,X^n_s,\PP \circ (X^n_s)^{-1},a) - \sigma(s,X_s,m^x_s,a)\right|^2\Lambda^n_s(da)ds\right] \\
	&\le C\int_0^t\left(\|X^n-X\|_s^2 + \ell_{s,\C^d,2}^2(\PP \circ (X^n)^{-1},m^x)\right)dt.
\end{align*}
By Gronwall's inequality,
\[
\E\left[\|X^n-X\|_t^2\right] \le C\int_0^t\ell_{s,\C^d,2}^2(\PP \circ (X^n)^{-1},m^x)ds + C\E[I_2 + I_4].
\]
Since $\PP \circ (X^n,X)^{-1}$ is a coupling of $\PP \circ (X^n)^{-1}$ and $m^x$, we have
\[
\ell_{t,\C^d,2}^2(\PP \circ (X^n)^{-1},m^x) \le \E\left[\|X^n-X\|_t^2\right],
\]
and another application of Gronwall's inequality yields
\begin{align*}
\ell_{\C^d,2}^2(\PP \circ (X^n)^{-1},m^x) &\le C\E[I_2+I_4].
\end{align*}
Assumption (A.4) and $m^x = \PP \circ X^{-1}$ imply (since $p < p'=2$)
\[
|b(s,X_s,m^x_s,a)|^2 \le C\left(1 + \|X\|^2 + |a|^2\right),
\]
and it follows from \eqref{pf:valueequal1.1} and continuity of $b$ that $\E[I_2] \rightarrow 0$ as $n\rightarrow\infty$. Similarly, assumption (A.4) implies (since $p_\sigma \le 2$)
\[
|\sigma(s,X_s,m^x_s,a)|^2 \le C\left(1 + \|X\|^2 + |a|^2\right),
\]
and it follows from \eqref{pf:valueequal1.1}, continuity of $\sigma$, and Lemma \ref{le:meleard} that $\E[I_4] \rightarrow 0$ as $n\rightarrow\infty$. We finally conclude that
\begin{align*}
\lim_{n\rightarrow\infty}\E[\|X^n-X\|^2] = 0,
\end{align*}
and so $\PP \circ (X^n,\Lambda^n)^{-1} \rightarrow m$ in $\P^p(\C^d\times\V)$.
It follows from Lemma \ref{le:gamma-usc} and \eqref{pf:valueequal-estimate1.0} that $\Gamma(\PP \circ (X^n,\Lambda^n)^{-1}) \rightarrow \Gamma(m)$.

\textbf{Step 2.} In light of step 1, we may now assume without loss of generality that our original control $\Lambda$ is bounded, in the sense that there exists $r > 0$ such that $\Lambda([0,T] \times B_r)=0$ a.s., where $B_r$ denotes the centered ball of radius $r$. We now use the chattering lemma \cite[Theorem 2.2(b)]{elkaroui-partialobservations} (originally due to Fleming \cite{fleming-generalized}) to find a sequence of progressively measurable $B_r$-valued processes $(\alpha^n_t)_{t \in [0,T]}$ such that $\delta_{\alpha^n_t}(da)dt \rightarrow \Lambda_t(da)dt$ a.s., and by another result of M\'el\'eard \cite{meleard1992martingale} we can find a $d_W$-dimensional Wiener process $W$ (again by extending the probability space) such that
\[
\lim_{n\rightarrow\infty}\E\left[\left(\int_0^T\varphi(t,\alpha^n_t)dW_t - \int_{B_r \times [0,T]}\varphi(t,a)M(da,dt)\right)^2\right] = 0,
\]
for every bounded predictable random $\R^{d \times d_W}$-valued function $\varphi$, where $M$ is again a vector $M=(M^1,\ldots,M^{d_W})$ of orthogonal martingale measures, each with intensity $\Lambda_t(da)dt$.\footnote{This does not follow immediately from Lemma \ref{le:meleard}. The key point is that we have a \emph{single} Wiener process $W$, whereas Lemma \ref{le:meleard} would yield a \emph{sequence} $W^n$.} As in Lemma \ref{le:meleard}, we can relax the boundedness assumption, as long as there is an integrable random variable $Z \ge 0$ such that $|\varphi(t,a)| \le Z$ a.s. Define $X^n$ as the unique solution of the McKean-Vlasov SDE
\[
dX^n_t = b(t,X^n_t,\PP \circ (X^n_t)^{-1},\alpha^n_t)dt + \sigma(t,X^n_t,\PP \circ (X^n_t)^{-1},\alpha^n_t)dW_t, \ X^n_0=X_0.
\]
We prove exactly as in Step 1 that $\E\|X^n-X\|^2 \rightarrow 0$, and so $\PP \circ (X^n,\delta_{\alpha^n_t}(da)dt)^{-1} \rightarrow \PP \circ (X,\Lambda)^{-1} = m$ in $\P^p(\C^d\times\V)$. Because the controls are uniformly bounded by $r$, the uniform integrability condition of Lemma \ref{le:gamma-usc} holds trivially, and we conclude that $\Gamma(\PP \circ (X^n,\delta_{\alpha^n_t}(da)dt)^{-1}) \rightarrow \Gamma(m)$.

\textbf{Step 3.} In light of step 2, we now assume without loss of generality that our original $m$ and our filtered probability space supports process $(X,\alpha,W)$ satisfying:
\begin{enumerate}
\item $X$ is a $d$-dimensional $\FF$-adapted process satisfying.
\item $\alpha$ is an $\FF$-predictable $A$-valued process, uniformly bounded in norm by a constant $r > 0$.
\item $W$ is a $d_W$-dimensional Wiener process.
\item The McKean-Vlasov equation holds:
\begin{align}
dX_t = b(t,X_t,m^x_t,\alpha_t)dt + \sigma(t,X_t,m^x_t,\alpha_t)dW_t. \label{pf:valueequal0-2}
\end{align}
\item The law of $(X,\delta_{\alpha_t}(da)dt)$ is precisely $m$.
\end{enumerate}
The final step is to approximate $\alpha$ in a weak sense by controls which are \emph{strong}, i.e., progressively measurable with respect to the filtration $\FF^W=(\sigma(X_0,W_s : s \le t))_{t \in [0,T]}$ generated by the Wiener process and initial state.
For this we appeal to \cite[Lemma 3.11]{carmonadelaruelacker-mfgcommonnoise} to find a sequence of $\FF^W$-progressively measurable $A$ processes  $(\alpha^n_t)_{t \in [0,T]}$ which share the same uniform bound as $\alpha$, such that
\begin{align}
\PP \circ (dt\delta_{\alpha^n_t}(da),W)^{-1} \rightarrow \PP \circ (\delta_{\alpha_t}(da)dt,W)^{-1}, \ \text{ in } \P^p(\V \times \C^{d_W}). \label{def:strongapprox}
\end{align}
In other words, $\alpha^n$ are \emph{strong controls} which approximate $\alpha$ in joint law.
Thanks to assumption \ref{assumption:B}, there exists a unique strong solution $X^k$ of the McKean-Vlasov SDE
\[
dX^n_t = b(t,X^n_t,\PP \circ (X^n_t)^{-1},\alpha^n_t)dt + \sigma(t,X^n_t,\PP \circ (X^n_t)^{-1})dW_t, \quad\quad X^n_0=X_0.
\]
Define $m^n = \PP \circ (X^n,dt\delta_{\alpha^n_t}(da))^{-1}$, and note that $m^n$ belongs to the set $\RC^s$ of strong controls. Now, suppose that we can show $m^n \rightarrow m$ in $\P^p(\C^d\times\V)$. Because $\alpha^n$ are uniformly bounded, it will then follow from Lemma \ref{le:estimate-optimal} that
\[
\sup_n\E\left[\|X^n\|^{p''}\right] < \infty,
\]
where we recall that $p'' > p'=2$  and $\int|x|^{p''}\lambda(dx) < \infty$ by assumption.
This is enough to verify the uniform integrability condition of Lemma \ref{le:gamma-usc}, which shows that $\Gamma(m^n) \rightarrow \Gamma(m)$. Hence, it remains to prove that $m^n \rightarrow m$ in $\P^p(\C^d\times\V)$.

It follows from Proposition \ref{pr:itocompactness} and the uniformly boundedness of $\alpha^n$ that the sequence $Q_n := \PP \circ (X^n,dt\delta_{\alpha^n_t}(da),W)^{-1}$ is precompact in $\P^p(\C^d \times \V \times \C^{d_W})$. Now let $(\overline{X},\overline{\Lambda},\overline{W})$ denote the canonical process on $\C^d \times \V \times \C^{d_W}$. If $Q$ denotes any limit point of $Q_n$, it is clear from \eqref{def:strongapprox} that 
\begin{align}
Q \circ (\overline{\Lambda},\overline{W},\overline{X}_0)^{-1} = \lim_{n\rightarrow\infty}\PP \circ (\delta_{\alpha^n_t}(da)dt,W,X_0)^{-1} = \PP \circ (\delta_{\alpha_t}(da)dt,W,X_0)^{-1}. \label{pf:inputlaws-equal}
\end{align}
Thus, under $Q$, there exists a $\sigma(\overline{X}_0,\overline{W}_s : s \le t)$-progressively measurable $A$-valued process $\overline{\alpha}$ such that $\overline{\Lambda}_t = \delta_{\overline{\alpha}_t}$ for a.e. $t$, almost surely.
We then argue using continuity of the coefficients $(b,\sigma)$ (either using martingale problems or the results of Kurtz and Protter \cite{kurtzprotter-weakconvergence}) that under $Q$ the following SDE holds:
\[
d\overline{X}_t = b(t,\overline{X}_t,\PP \circ \overline{X}_t^{-1},\overline{\alpha}_t)dt + \sigma(t,\overline{X}_t,\PP \circ \overline{X}_t^{-1},\overline{\alpha}_t)d\overline{W}_t.
\]
Because of \eqref{pf:inputlaws-equal}, uniqueness in law of the SDE \eqref{pf:valueequal0-2} implies $Q = \PP \circ (X,\delta_{\alpha_t}(da)dt,W)^{-1}$. (See \cite{kurtz-yw2013}, particularly example 2.14 therein, for a more careful discussion of the Yamada-Watanabe theorem in the context of McKean-Vlasov equations.) This holds for every limit point $Q$, and we conclude that
\[
\PP \circ (X^n,dt\delta_{\alpha^n_t}(da),W)^{-1} \rightarrow \PP \circ (X,\delta_{\alpha_t}(da)dt,W)^{-1}, \ \text{ in } \P^p(\C^d\times \V \times \C^{d_W}).
\]
Marginalizing yields
\[
m^n = \PP \circ (X^n,dt\delta_{\alpha^n_t}(da))^{-1} \rightarrow \PP \circ (X,\delta_{\alpha_t}(da)dt)^{-1} = m, \ \text{ in } \P^p(\C^d\times \V).
\]
\hfill\qedsymbol

\appendix

\section{A note on the filtration of $\C^d\times\V$} \label{ap:separating}

The goal of this section is to clarify a technical point that arose in the proofs of Lemma \ref{le:rceps-compact} and Proposition \ref{pr:tight-limits}.
Recall that $(X,\Lambda)$ denotes the identity map on $\C^d\times\V$. Recall also that the natural filtration on $\C^d\times\V$ is defined as 
\[
\F_t  = \sigma(X_s,\Lambda([0,s]\times B) : s \le t,  B \subset A \text{ Borel}).
\]
If $E$ is a Polish space and $\G$ a sub-$\sigma$-field of the Borel sets, let us say  a family $\Phi$  of $\G$-measurable functions is a \emph{separating class for $\G$} if $\int\varphi\,d\mu=\int\varphi\,d\nu$ for all $\varphi \in \Phi$ implies $\mu=\nu$ on $\G$, whenever $\mu,\nu \in \P(E)$. When $\G$ is the entire Borel $\sigma$-field, simply say that $\Phi$ is a \emph{separating class}. It is known that every Polish space admits a countable separating class consisting of bounded continuous functions. Note that if $\Phi$ is separating then $\int\varphi\,d\mu = \int\varphi\,d\nu$ for all $\varphi \in \Phi$ implies $\mu=\nu$ for all bounded signed measures $\mu,\nu$, which is spanned by the space of probability measures. In particular, if $\Phi$ is separating for $\G$, and if $X$ and $Y$ are random variables satisfying $\E[X\varphi] = \E[Y\varphi]$ for all $\varphi \in \Phi$, then $\E[X | \G] = \E[Y | \G]$ a.s.

\begin{lemma} \label{le:filtration-separating}
For each $t \in [0,T]$, $\F_t$ admits a countable separating class $\Phi_t$ of bounded continuous functions.
\end{lemma}
\begin{proof}
The statement is clearly true for $t=0$, as $\F_0 = \sigma(X_0)$ can be identified with the Borel $\sigma$-field of $\R^d$. Fix $t > 0$. We deal with $X$ and $\Lambda$ separately. For $X$, notice that the restriction of $x \in \C^d$ to $x_{\cdot \wedge t} \in C([0,t];\R^d)$ is a continuous operation and that the Polish space $C([0,t];\R^d)$ admits a countable separating class $\Phi^X_t$ of bounded continuous functions.
For $\Lambda$, define first the restriction of any $q \in \V$ by $q^t(\cdot) = t^{-1}q(([0,t] \times A) \cap \cdot) \in \P^p([0,t] \times A)$. The map $q \mapsto q^t$ is continuous from $\V$ to $\P^p([0,t] \times A)$ (see, e.g. \cite[Corollary A.3]{lacker-mfgcontrolledmartingaleproblems}). As a Polish space, $\P^p([0,t] \times A)$ admits a countable separating class $\Phi^\Lambda_t$ of bounded continuous functions. Without loss of generality, assume that the constant function $1$ belongs to both $\Phi^X_t$ and $\Phi^\Lambda_t$. Finally, define $\Phi_t$ to be the set of functions of the form $\C^d \times \V \ni (x,q) \mapsto \varphi(x_{\cdot \wedge t})\psi(q^t)$, where $\varphi \in \Phi^X_t$ and $\psi\in \Phi^\Lambda_t$. Then $\Phi_t$ fits the bill; see Proposition 3.4.6 of \cite{ethierkurtz} for a proof that $\Phi_t$ is separating. 
\end{proof}

\bibliographystyle{amsplain}
\bibliography{MFGconvergence-bib}

\end{document}